\documentclass[10pt]{amsart}
\usepackage{ amsthm, amssymb, hyperref, a4wide,
todonotes, url}

\newtheorem{thm}[subsection]{Theorem}
\newtheorem{prop}[subsection]{Proposition}
\newtheorem{cor}[subsection]{Corollary}
\newtheorem{lem}[subsection]{Lemma}

\theoremstyle{definition}
\newtheorem{notation}[subsection]{Notation}
\newtheorem{defi}[subsection]{Definition}
 
\newcommand{\La}{\Lambda}
\newcommand{\vx}{\operatorname{vx}}
\newcommand{\f}{{\operatorname f}}
\newcommand{\rep}{{\operatorname {rep}}}
\newcommand{\w}{\widetilde}
\newcommand{\wh}{\widehat}
\renewcommand{\o}{\overline}
\newcommand{\la}{\lambda}
\newcommand{\Cl}{\operatorname{cc}}
\newcommand{\Bl}{\operatorname{Bl}}
\newcommand{\Irr}{\operatorname {Irr}}
\newcommand{\Syl}{\operatorname {Syl}}
\newcommand{\calP}{\mathcal P}
\newcommand{\calD}{\mathcal D}
\newcommand{\calQ}{\mathcal Q}
\newcommand{\calS}{{\mathcal S}}
\newcommand{\Cent}{\operatorname C}
\newcommand{\Z}{\operatorname Z}
\newcommand{\dz}{\operatorname{dz}}
\newcommand{\bl}{\operatorname{bl}}

\newcommand{\NNN}{\operatorname {N}}
\newcommand{\Aut}{\operatorname {Aut}}
\newcommand{\Inn}{\operatorname {Inn}}
\newcommand{\TT}{\mathbb T}
\newcommand{\IBr}{\operatorname {IBr}}

\newcommand{\CC}{\mathbb C}
\newcommand{\ZZ}{\mathbb Z}
\newcommand{\GL}{\operatorname {GL}}

\newcommand{\mathcalR}{\mathcal R}
\def\spann<#1>{\left\langle#1\right\rangle}
\newcommand{\Rad}{\operatorname{Rad}}

\newcommand{\Aphi}{{A(\phi)}}

\newcommand{\Ephi}{\w\phi}

\newcommand{\EOphi}{\w\phi'}
\newcommand{\mathcalO}{\mathcal O}
\newcommand{\Mat}{\operatorname{Mat}}

\newcommand{\rank}{ {\mathrm{rank}}_{\mathcal O} }

\newcommand{\enumroman}{\renewcommand{\labelenumi}{(\roman{enumi})}   
\renewcommand{\theenumi}{\thesubsection(\roman{enumi})}}
\newcommand{\enumalph}{\renewcommand{\labelenumi}{(\alph{enumi})} 
\renewcommand{\theenumi}{\thesubsection(\alph{enumi})}}

\title[The inductive AM and BAW conditions for blocks with cyclic defect]{
The inductive Alperin-McKay and blockwise Alperin weight 
conditions for blocks with cyclic defect groups}

\author{Shigeo Koshitani and Britta Sp\"ath}
\thanks{
The first and second authors have been supported, respectively,
by the Japan Society for Promotion of Science (JSPS), 
Grant-in-Aid for Scientific Research (C)23540007, 2011--2014,
 the DFG Priority Programm, SPP 1388 and the ERC Advanced Grant 291512.       
}
\address{Department of Mathematics, Graduate School of Science, 
Chiba University, 1-33 Yayoi-cho, Inage-ku, Chiba, 263-8522, Japan}
\email{koshitan@math.s.chiba-u.ac.jp}
\address{FB Mathematik, TU Kaiserslautern, Postfach 3049, 
67653 Kaiserslautern, Germany.}
\email{spaeth@mathematik.uni-kl.de}

\begin{document}

\begin{abstract}
We verify the inductive blockwise Alperin weight 
(BAW)    
and 
the inductive   
Alperin-McKay
(AM)   
conditions introduced by the second author for blocks 
of finite quasisimple groups with cyclic defect 
groups. Furthermore we establish a criterion that describes 
conditions under which the inductive AM condition 
for blocks with abelian defect groups
implies the inductive BAW condition for those blocks.
\end{abstract}

\maketitle

\section{Introduction}
Two of the most important counting conjectures 
in the representation theory of finite groups are 
the Alperin-McKay and the Alperin weight conjectures concerned with invariants of blocks. 
While the conjectures are known only for specific classes of 
blocks of finite groups, there is now a new approach that might 
lead to general proofs of the conjectures.

In \cite{Spaeth_red_BAW} and \cite{Spaeth_AM_red} 
it is shown that these conjectures hold if every non-abelian 
simple group satisfies the inductive blockwise Alperin weight (BAW) condition 
and the inductive Alperin McKay (AM) condition. 
In the case of the Alperin weight conjecture 
the proof also leads to an approach for blocks with a given defect group. 
Recall that for any finite group $G$ and a block $B$ of $G$ with defect group $D$ 
the blockwise Alperin weight conjecture holds, 
if the inductive blockwise Alperin weight (BAW) condition holds 
with respect to $ \mathcalR_D$ for all simple groups involved in $G$, 
where $\mathcalR_D$ is the set of all finite groups involved in $D$, 
see Theorem B of \cite{Spaeth_red_BAW}. 
(A group $D_1$ is {\it involved in $D$} if there exist groups 
$H_1\lhd H_2\leq D$ such that $H_2/H_1$ is isomorphic to $D_1$.) 

Analogously one can establish blockwise versions of 
the inductive Alperin-McKay (AM) condition, see 
Definitions \ref{def4_2} and \ref{def_ind_AM_block} below. 
We verify those conditions for blocks with cyclic defect groups, 
and also verify the inductive BAW condition for nilpotent blocks.

In a few cases the inductive conditions have been checked 
for specific finite simple groups and with respect to certain $p$-groups, 
see \cite{Sc12}, \cite{Spaeth_AM_red}, \cite{Spaeth_red_BAW}, and \cite{CS12}. 
The inductive BAW condition for a prime $p$ could be verified for most 
sporadic finite simple groups, 
the alternating groups and finite simple groups of Lie type 
defined over a field of characteristic $p$, see 
\cite{Breuer}, \cite{Malle_AWC_Alt}, \cite{Spaeth_red_BAW}, and \cite{Sc12}. 
In the present paper we prove that the inductive blockwise Alperin 
weight (BAW) condition and the inductive Alperin-McKay (AM) condition hold
for all blocks of finite quasisimple groups with cyclic defect groups.

While in most of the anterior proofs 
the knowledge on the representation theory of the specific 
quasisimple group was a key,
we apply here (well-)known results on blocks with cyclic defect 
groups for the proof. 
  
\begin{thm} \label{thm_BAW_cycl} \label{thm1_1}
Let $S$ be a finite non-abelian simple group, $G$ its universal covering group and $\o G$ its universal $p'$-covering group.
Then the inductive Alperin-McKay (AM) condition from Definition \ref{def_ind_AM_block} 
holds for all blocks of $G$ with cyclic defect groups and the
inductive blockwise Alperin weight (BAW) condition from Definition \ref{def4_2} 
holds for every  block of $\o G$ with cyclic defect groups. 
\end{thm}

This extends the earlier results from \cite[Corollary 8.3(b)]{Spaeth_AM_red} and \cite[Proposition 6.2]{Spaeth_red_BAW} on blocks with cyclic defect groups, where in addition the outer automorphism groups had to be cyclic.
In Corollary \ref{cor_explizit} we give a list of cases where for a prime $p$ and
a finite simple group $S$ of Lie type, the mentioned inductive conditions hold.
We suspect that Theorem \ref{thm_BAW_cycl} implies the
inductive condition of Dade's conjecture recently proposed by the second author, see \cite{Spaeth_red_Dade}. 
The result of \cite{Marcus} on Dade's inductive Conjecture
for blocks with cyclic defect groups seems related but does not directly 
provide the more involved requirements of the inductive AM condition.

The main part of the proof is devoted to prove that the 
inductive Alperin-McKay condition holds, especially the condition on the existence of 
certain extensions. Applying the following criterion we show that also the inductive blockwise Alperin weight condition holds for the same block.

\begin{thm}\label{thm_inductiveAM_inductiveBAW}\label{thm1_2}
Let $S$ be a finite non-abelian simple group, $G$ its 
universal covering group and $B$ a $p$-block of $G$ with abelian defect group $D$. 
Assume that $B$ and $D$ satisfy the following:
 \begin{enumerate}
 \item[(i)] The inductive AM condition holds for $B$ with 
a group $M := \NNN_G(D)$ and a bijection 
$\Lambda_B^G: \Irr_0(B)\rightarrow \Irr_0(B')$, 
where $B'\in\Bl(M)$ with $(B')^G=B$.

 \item[(ii)] The decomposition matrix $C'$ associated with 
$\calS:=(\Lambda_B^G)^{-1}(\{ \chi'\in\Irr(B') 
\mid D\leq \ker(\chi')\})$ 
is a unitriangular matrix after suitable ordering of 
the characters, where  $B'\in\Bl(M)$ is the Brauer correspondent  of $B$. 
 \end{enumerate}  
Let $\o B$ be the unique $p$-block of the universal 
$p'$-covering group of $S$ dominated by $B$. 
Then the inductive BAW condition holds for $\o B$.
\end{thm}
Exploiting the fact that the defect groups of nilpotent blocks of quasisimple groups are abelian according to An and Eaton,
we prove in addition that the inductive BAW conditions holds 
for nilpotent blocks.

\begin{thm}\label{thm1_3}
The inductive BAW condition 
holds for nilpotent blocks of finite quasisimple groups.
\end{thm}

A main ingredient is the study of blocks with cyclic defect groups 
due to various authors. Nevertheless the existing results does not allow
an immediate approach to
Condition (iii) of the inductive AM condition from Definition \ref{def_ind_AM_block}, hence 
more considerations are required. 
Furthermore we use 
\cite{KoshitaniSpaeth1} to simplify the checking of the 
extendibility part of the conditions.
This provides an example of how the conditions, especially 
the last technical part on the existence of certain extensions 
with additional properties, can be verified. 
In the end we illustrate the use of Theorem \ref{thm_BAW_cycl} by listing some simple groups 
of Lie type where  the inductive conditions hold. 

This article is structured in the following way: We introduce
the main notation 
in Section \ref{sec_not}. In Section \ref{sec4} we prove 
the inductive BAW condition for nilpotent blocks and thereby give the proof of
Theorem \ref{thm1_3}. 
After recalling various results about blocks with cyclic
defect groups in Section \ref{sec5} we extend them
in Section \ref{sec5_beyond} in order to determine the Clifford theory of 
characters in blocks with cyclic defect groups
by means of local data. We use that in Section \ref{sec6} 
to verify the inductive AM condition for those blocks. 
We finish in Section \ref{sec7} by proving Theorem \ref{thm1_2}
and thereby a criterion when the inductive 
AM condition implies the inductive BAW condition.
We conclude with the proof of the inductive BAW 
condition for blocks with cyclic defect groups 
as an application of Theorem \ref{thm1_2}.

\section{Notation}\label{sec_not}
In this section we explain most of the used notation. 
For characters and blocks we use mainly the notation of 
\cite{NagaoTsushima}, and \cite{Navarro} respectively. 

Let $p$ be a prime. Let $(\mathcal K, \mathcalO, k)$ be a 
$p$-modular system, that is "big enough" with respect to 
all finite groups occurring here. That is to say, 
$\mathcalO$ is a complete discrete valuation ring of
rank one such that its quotient field $\mathcal K$ is 
of characteristic zero, 
and its residue field $k=\mathcal O/\mathrm{rad}(\mathcal O)$ is of
characteristic $p$, and that $\mathcal K$ and $k$ are
splitting fields for all finite groups occurring in this paper.
We denote the canonical epimorphism from $\mathcal O$ to $k$ 
by $^*: \mathcal O\rightarrow k$. 

In this paper all considered groups are finite. For a finite group $G$
we denote by $\Bl(G)$ the set of $p$-blocks of $G$. 
For $N\lhd G$ and $b\in \Bl(N)$ and we denote by $\Bl(G\mid b)$ 
the set of $p$-blocks of $G$ covering $b$. We write also 
$\Bl(G \mid D)$ for the set of $p$-blocks of $G$ with defect group $D$.  

For a character $\phi\in\IBr(G)\cup \Irr(G)$ we denote by $\bl(\phi)$ 
the $p$-block $\phi$ belongs to. For $N\lhd G$ we say that 
$\o B\in\Bl(G/N)$ is {\it dominated} 
by $B\in\Bl(G)$ and write 
$\o B\subseteq B$, if all irreducible characters of $\o B$ lift to 
characters of $B$, see \cite[p.198]{Navarro} or 
\cite[p.360]{NagaoTsushima}.

We denote by $\dz (G)$ the set of defect zero 
irreducible ordinary characters of $G$.
When $Q$ is a $p$-subgroup of $G$ and $B \in \Bl(G)$,
we denote by $\dz (\NNN_G(Q)/Q, B)$ the set of defect zero 
characters $\o\chi\in \Irr(\NNN_G(Q)/Q)$ such that 
$\bl(\o\chi)$ is dominated by a $p$-block $B'\in\Bl(\NNN_G(Q))$ 
with $(B')^G = B$. 

By $\Cl_G(x)$ we denote the conjugacy class of $G$ containing 
$x \in G$, and for any subset $X\subseteq G$ we denote by $X^+$ 
the sum $\sum_{x\in X} x$ in $\mathcal O G$ or $k G$, respectively. 
We write $G^0$ for the set of all $
p$-regular elements of $G$. The restriction of a character 
$\chi$ of $G$ to $G^0$ is denoted by $\chi^0$. 
For $B\in\Bl(G)$ we denote by 
$\la_{B}: \Z(kG)\rightarrow k $ the associated central function,  see 
page 48 of \cite{Navarro}. For $\phi\in\IBr(B)$ and $\chi\in\Irr(B)$ 
we write also $\la_\phi$ and $\la_\chi$ instead of $\la_B$. If $H$ is 
a subgroup of a finite group $G$ and if $\chi$ and $\nu$ are 
characters of $G$ and $H$ respectively, then we denote by $\chi_H$ and 
$\nu^G$ the restriction of $\chi$ to $H$ and 
the induction of $\nu$ to $G$, respectively.
By $\Irr(G\mid \nu)$ we denote the set of irreducible
constituents of $G$ and if $H\lhd G$ we write $\Irr(H\mid \chi)$
for the constituents of $\chi_{H}$.
For a group $A$ acting on a set $X$ we denote by $A_{X'}$ the 
stabilizer of $X'$ in $A$, where $X'\subseteq X$.

Later we use also vertices and other methods seeing blocks as algebras. 
The thereby used notation will be introduced later.

\section{The inductive bockwise Alperin weight condition for nilpotent blocks}\label{sec4}
In this section we introduce a condition on blocks of quasisimple groups, 
that partitions the inductive blockwise Alperin weight condition (BAW condition for short) in Definition \ref{def4_2} 
and prove that the inductive BAW condition holds 
for all nilpotent blocks of finite quasisimple groups. 

Recall that the {\it universal $p'$-covering group} $G$ of 
a finite perfect 
group $S$ is the quotient $Y/\Z(Y)_{p}$ where $Y$ is the universal 
covering group of $S$ and $\Z(Y)_{p}$ the Sylow $p$-subgroup of 
$\Z(Y)$. (Recall that the universal covering group of a perfect group $S$
is a perfect group $Y$ such that $Y/\Z(Y)\cong S$ and $|\Z(Y)|$ is 
maximal.)

In the following let $\mathcal R$ be a certain set of $p$-groups.
We denote by $\Bl_\mathcalR(G)$ the set of all $p$-blocks of $G$ 
having a defect group in $\mathcal R$, and $\IBr_\mathcal R(G)$ 
the set of the irreducible Brauer characters belonging to a block in $\Bl_{\mathcal R}(G)$. 
For example $\Bl_{\{1\}}(G)$ denotes the set of defect zero blocks of 
$G$, and $\IBr_{\{1\}}(G)$ is then the restriction of defect zero 
characters of $G$. For a $p$-subgroup $Q\leq G$ we 
denote by $\dz(\NNN_G(Q)/Q, \Bl_\mathcal R(G) )$ characters 
$\o \chi\in\dz(\NNN_G(Q)/Q)$ whose lift $\chi\in\Irr(\NNN_G(Q))$ 
satisfies $ \bl(\chi)^G\in  \Bl_\mathcal R(G)$. In the following let 
$\mathcal R_c$ be the set of all finite cyclic $p$-groups.

In \cite{Spaeth_red_BAW} two versions of the inductive BAW condition 
are presented, one general version and a version relative to a 
set of finite $p$-groups, see Definitions 4.1 and 5.17 of 
\cite{Spaeth_red_BAW}. Additionally we present here a blockwise 
version of the condition and give its relation to the established 
conditions.

\begin{notation}
We denote by $\Rad(G)$ the set of radical $p$-subgroups of a finite group $G$, 
and $\Rad(G)/\sim_G$ denotes an arbitrary $G$-{\it transversal} 
of $\Rad(G)$. Recall, for a group $A$ acting on $G$ we denote by 
$A_B$ the stabilizer of $B\in \Bl(G)$.
For $N\lhd G$,  $\phi\in\IBr(N)$ and $\psi\in\IBr(G)$
we denote by $\IBr(G\mid \phi)$ the set of irreducible    
constituents of $\phi^G$, and by $\IBr(N\mid \psi)$  the set
of irreducible constituents of $\psi_N$.
\end{notation}

\begin{defi}[Inductive BAW condition for a $p$-block,
see also {\cite[Definition 5.17]{Spaeth_red_BAW}}]    
\label{def_BAWC_ind_Cond}\label{def4_2}\label{def_BAWC}
Let $p$ be a prime, $S$ a finite non-abelian simple group, 
$G$ the universal $p'$-covering group of $S$, and $B\in\Bl(G)$. 
We say that the {\it inductive BAW condition holds for $B$}, 
if the following statements are satisfied:
\enumroman
\begin{enumerate}
\item \label{def_BAWC_part}
There exist subsets 
$\IBr(B \mid Q) \subseteq \IBr(B)$ for $Q\in\Rad(G)$  with

\begin{enumerate}
\item 
$\IBr(B\mid Q)^a = \IBr(B\mid Q^a)$ for every $a \in \Aut(G)_B$, and
\item 
\[\IBr(B)=\bigcup_{Q\in \Rad(G)/\sim_G}^. \IBr(B\mid Q)
\qquad \text{ (disjoint union)}.\]
\end{enumerate}
\item \label{def_BAWC_bij} 
For every $Q\in \Rad(G)$ there exists a bijection 
\[ \Omega_Q^G: \IBr(B\mid Q)\longrightarrow \dz(\NNN_G(Q)/Q , B) ,\]
such that $\Omega_Q^G(\phi)^a= \Omega_{Q^a}^G(\phi^a)$ 
for every $a\in \Aut(G)_B$ and $\phi\in \IBr(B\mid Q)$.
\item \label{def_BAWC_c}
For every $Q \in \Rad(G)$ and every character 
$\phi\in \IBr(B \mid Q)$, 
there exist a finite group $A(\phi)$ 
and characters $\Ephi\in \IBr(\Aphi)$ 
and $\EOphi \in \IBr(\NNN_\Aphi(\o Q))$ 
where $\o Q:= QZ/Z$ and $Z := \Z(G) \cap \ker \phi$,
with the following properties:
\begin{enumerate}
\item \label{def_BAWC_c_group}
For  $\o G:= G/Z$ 
the group $A:=\Aphi$ 
satisfies $\o G\lhd A$, $A/\Cent_{A}(\o G)\cong \Aut(G)_\phi$, 
$\Cent_A(\o G)=\Z(A)$ and $p\nmid |\Z(A)|$.
\item 
$\Ephi \in \IBr(A)$ is an extension of $\o \phi\in\IBr(\o G)$, 
that is associated with $\phi$, i.e. $\o \phi$ lifts to $\phi$.
\item 
The character $\EOphi \in \IBr(\NNN_A(\o Q))$ is an extension 
of the character of $\NNN_{\o G}(\o Q)$ associated with 
$\Omega_Q^G(\phi)^0\in \IBr(\NNN_{G}(Q)/Q)$, i.e. $\EOphi$ lifts to $\Omega_Q^G(\phi)^0$.
\item \label{def_BAWC_cohom_block}
For every $J$ with 
$\o G\leq J\leq A$ 
the characters $\Ephi$ and $\EOphi$ satisfy 
 \[\bl(\Ephi_{J})=\bl\Big((\EOphi)_{\NNN_J(\o Q)}\Big)^{J}.\] 
\end{enumerate}
\item \label{def_BAWC_triv}
If $B$ is of defect zero, then 
$\Omega_1^G(\chi^0)=\chi$ for every $\chi\in \Irr(B)$ and 
$\Ephi=\EOphi$ for every $\phi\in \IBr (B\mid 1)$.
\end{enumerate}
\enumalph
\noindent(Note that Condition (iii)(c) above makes sense since for every $Q\in\Rad(G)$ and $\phi\in\IBr(B\mid Q)$
the sets $\IBr(\Z(G)\mid \phi)$ and $\IBr(\Z(G)\mid \Omega_Q^G(\phi)^0)$  coincide.
This follows from the fact that 
 $\bl(\phi)$ and $\bl(\Omega_Q^G(\phi))$ cover the same block of the $p'$-group $Z(G)$ because of $\bl(\phi)=\bl(\Omega_Q^G(\phi))^G$.)
\end{defi}

This definition of the inductive BAW condition for a block 
is a partitioning of the one established in 
\cite[Definition 4.1]{Spaeth_red_BAW}.

\begin{lem}
Let $S$ be a finite non-abelian simple group, 
$G$ its universal $p'$-covering group 
and $\mathcal R$ a set of $p$-groups. Assume that for some 
$\Aut(G)$-transversal
in $\Bl_\mathcal R(G)$ the inductive BAW 
condition holds. Then the inductive BAW condition from 
Definition 5.17 of \cite{Spaeth_red_BAW}  
holds for $S$ with respect to $\mathcal R$.
\end{lem}

\begin{proof} Note that following the explanation to Condition (iii) for every $Q\in \mathcal R$,
every character of 
$\Omega_Q^G\left (\IBr(B\mid Q)\cap \IBr(G\mid \nu^0)\right)$ 
lifts to a character in $\Irr(\NNN_G(Q) \mid \nu)$ 
for every $\nu\in \Irr(\Z(G))$, since the block $B$ and 
the block $\bl(\phi')$ covers the same block of the $p'$-group 
$\Z(G)$, whenever $\phi'$ is a lift of a character in 
$\dz(\NNN_G(Q)/Q,B)$. Apart from this requirement
both conditions coincide. 
\end{proof}

When verifying the above condition for any block,
Part (iii) of Definition \ref{def_BAWC} is crucial. 
The group $\Aphi$ required in this condition can be 
constructed for all characters $\phi\in\IBr(G)$. 

\begin{lem}\label{rem_exist_A}
Let $p$ be a prime, $S$ a finite non-abelian simple group, $G$ 
the universal $p'$-covering group of $S$, and $\phi\in\IBr(G)$.
Then there exists a finite group $A$ which satisfies the following: 
\enumroman
\begin{enumerate}
\item[(i)]  $\o G\lhd A$ with $\o G:=G/(\Z(G)\cap \ker\phi)$, 
$A/\Cent_{A}(\o G)\cong \Aut(G)_\phi$, $\Cent_A(\o G)=\Z(A)$, 
and $p\nmid |\Z(A)|$.
\item[(ii)] The character $\o\phi\in \IBr(\o G)$ associated with 
$\phi$ extends to $A$.
\end{enumerate}
\end{lem}

\begin{proof}
In a first step we construct a projective $k$-representation 
$\mathcal P$ of $\Aut(G)_\phi$ and then determine 
a central extension of $\Aut(G)_\phi$ using the factor set of 
$\mathcal P$. Finally we prove that the thereby obtained finite 
group has the properties claimed in the statement.

Let $\o\phi\in\IBr(\o G)$  be the Brauer character, that lifts to 
$\phi$, and $\mathcal D$ a $k$-representation of $G$ associated with 
$\phi$. We denote by $\o \calD$ the $k$-representation of $\o G$ associated with $\o\phi$, such that $\mathcal D$ is the lift of $\o  {\mathcal D}$. 
Let $\TT$ be a full representative system of  
$\Inn(G)$-cosets in $\Aut(G)_\phi$. For $t \in\TT$ 
we define $\calP(t)$ by the following: 
Let $Y := \o G\rtimes \spann<t>$. Then there exists 
an extension $\w\calD$ of $\calD$ to $Y$. 
We set $\calP(t):=\w\calD(t)$. 
Further we choose an $\Z(\o G)$-section $\rep:\Inn(G)\rightarrow \o G$, i.e. a map $\rep: \Inn(G) \rightarrow \o G$ 
such that for $x\in\Inn(G)= G/\Z(G)$ the element $\rep(x)$ induces the automorphism $x$ on $G$ via conjugation and $\rep(1_{\Inn(G)})=1_G$. 
We obtain a projective $k$-representation of $\Inn(G)$ by
\[ \calP(i)=\calD(\rep(i)) \text{ for every } i \in \Inn(G).\]
Let $\mathcal P:\Aut(G)_\phi \longrightarrow \GL_{\o\phi(1)}(k)$ be given by 
\[ 
\calP(i t)=\calD(\rep(i))\calP(t) 
\text{ for every } i \in \Inn(G) \text{ and } t\in \TT.
\]
Straight-forward calculations show that $\mathcal P$ is 
a projective representation. Let $\alpha$ be its factor set and  
$C$ the subgroup of $k^\times$ that is generated by 
$\alpha(a,a')\in C$ for $a,a'\in\Aut(G)_\phi$. By the construction 
of $\calP$ the values of $\alpha$ are finite roots of unity and hence 
$C$ is a finite group. 

Like in the proof of Theorem (8.28) of \cite{Navarro} 
the factor set $\alpha$ defines a central extension 
$A$ of $\Aut(G)_\phi$: 
The elements of $A$ are the pairs $(a,c)$ 
with $a\in \Aut(G)_\phi$ and $c\in C$, and multiplied by
\[ (a_1,c_1)(a_2,c_2)=(a_1a_2,\alpha(a_1,a_2)c_1c_2) 
\text{ for every }a_i\in \Aut(G)_\phi \text{ and } c_i\in C.\]

Let $\nu:\Z(\o G)\longrightarrow C$ be the morphism 
such that for $z\in \Z(\o G)$ the matrix $\calD(z)$ is 
a scalar matrix to $\nu(z)$. 
Then  $\o G$ is isomorphic to the normal subgroup of $A$, 
via an isomorphism $\rep(g)z\mapsto (g,z)$. 
This proves the property in (i). 

We have $\Cent_A(\o G) = C$. Accordingly 
$A/\Cent_A(\o G)\cong \Aut(G)_\phi$ and $\Cent_A(\o G)=C$. 
As $C$ is a finite subgroup of the multiplicative group of $k$, 
$p\nmid |C|$.
We observe that $\mathcal P$ lifts to a representation $\calQ$ of $A$, 
that is defined by 
\[ \calQ(a,c)=c\calP(a) \text{ for every } a 
\in \Aut(G)_\phi \text{ and }c\in C.\]
By straight-forward calculations we see that the Brauer 
character of $A$ afforded by $\calQ$ is an extension of $\o \phi$. 
\end{proof}

\begin{lem}\label{lem4_2}
Let $N\lhd G$, $L\lhd G$ with $N\leq L$ and $\w\phi\in\IBr(L)$ 
with $\phi := \w\phi_N\in\IBr(N)$. Assume that $\w\phi$ is 
$G$-invariant and that for every prime $q\neq p$ there exists 
an extension $\psi_q$ of $\phi$ to some $K_q\leq G$ with 
$(\psi_q)_{K_q\cap L}=(\w\phi)_{K_q\cap L}$, where $K_q$ satisfies  
$N\leq K_q$ and $K_q/N\in \Syl_q(G/N)$. Then $\w\phi$ extends to $G$. 
\end{lem}

\begin{proof} 
By Theorem (8.11) of \cite{Navarro} the character $\w\phi$ extends to 
$LK_p$, where $K_p\leq G$  satisfies $N\leq K_p$ and 
$K_p/N\in \Syl_p(G/N)$. According to  \cite[Theorem (8.29)]{Navarro}
it suffices to prove that $\w\phi$ extends to $LK_q$ for 
every prime $q\neq p$. Let $\calD$ be a $k$-representation of 
$L$ whose associated Brauer character is $\w\phi$.
We construct a $k$-representation of $LK_q$ extending 
$\calD$. According to Theorem (8.16) of \cite{Navarro}, for every prime 
$q\neq p$ there exists a $k$-representation $\calQ_q$ of $K_q$ 
associated to $\psi_q$ such that 
\[\calQ_q(x)=\calD(x) \text{ for every } x\in K_q\cap L.\]

Let $x\in K_q$. There exists a $k$-representation $\w\calD$ of 
$\calD$ to $\spann<L,x>$ that extends $\calD$. Both 
$(\calQ_q)_{\spann<L\cap K_q,x>}$ and $\w\calD_{\spann<L\cap K_q,x>}$ 
are extensions of $(\calQ_q)_{L\cap K_q}$. By Theorem (8.16) 
of \cite{Navarro} the matrices $\calQ_q(x)$ and $\w\calD(x)$ satisfy 
\[ \w\calD(x)=\zeta\calQ_q(x) \]
for some $\zeta\in k^\times$. Since $\calD(x)$ satisfies 
\[ \calD(l)^{\w\calD(x)}=\calD(l^x)\text{ for every } l\in L,\]
the matrix $\calQ_q(x)$ has the analogous property 
\[ \calD(l)^{\calQ_q(x)}=\calQ_q(l^x)\text{ for every } l\in L.\]
This implies that there exists a $k$-representation $\calD_q$ of 
$LK_q$ with 
\[ \calD_q(lk)=\calD(l)\calQ_q(k) \text{ for every } l
   \in L \text{ and } k\in K_q.\]
Accordingly $\w\phi$ can be extended to $LK_q$, and hence to $G$.
\end{proof}

We recall some properties of nilpotent blocks of quasisimple 
groups before we verify the inductive BAW condition for those blocks.

\begin{thm}[{An-Eaton {\cite[Theorem 1.1]{AnEaton1}} 
and {\cite[Theorem 1.1]{AnEaton2}}}]\label{thm_An_Eaton}
Let $S$ be a finite non-abelian simple group, 
$G$ a universal $p'$-covering group of $S$, 
and $B \in\Bl(G)$ a nilpotent block. 
Then $B$ has abelian defect groups. 
\end{thm}

Furthermore from the work of K\"ulshammer-Puig \cite{KuelshammerPuig} 
we can deduce the following about the extensibility of characters 
in nilpotent blocks. 

\begin{lem}\label{prop3_7}
Let $N\lhd G$ and $b\in\Bl(N)$ a $G$-invariant nilpotent block 
with defect group $D$. 
Let $b'\in\Bl(\NNN_N(D))$ be the Brauer correspondent of $b$. 
Assume that $p\nmid |G/N|$. 
If $\phi\in\IBr(b)$ extends to $G$, 
i.e. there exists $\w\phi\in\IBr(G)$ with $\w\phi_N = \phi$,
then for every $\phi'\in\IBr(b')$ 
there exists $\w\phi'\in\IBr(\NNN_G(D))$ with 
$(\w\phi')_{\NNN_N(D)} = \phi'$ such that

\[ \bl\Big( (\w\phi')_{\NNN_J(D)}\Big)^J=\bl(\w\phi_J) 
\text{ for every } J \text{ with } N\leq J \leq G.\]
\end{lem}

\begin{proof}
Set $H:= \NNN_G(D)$, $L:= \NNN_{G[b]}(D)$ and $M := \NNN_N(D)$.
Further, set $B := \bl(\w\phi)\in\Bl(G\mid b)$, 
$\mathfrak B := \bl(\w\phi_{G[b]})\in\Bl(G[b]\mid b)$,
$b' := \bl(\phi')\in\Bl(M)$. 
Let $B'\in\Bl(H)$ and $\mathfrak B'\in\Bl(L)$ 
be the Brauer correspondents of $B$ and $\mathfrak B$, respectively.
In the following we view blocks as bimodules.
 
Now, since $p\nmid |G[b]/N|$, it follows from
\cite[Lemma 3.4]{KoshitaniSpaeth1}
that ${(\mathfrak B){\downarrow}}{_{N\times N}} \, \cong b$ as 
$k[N\times N]$-bimodules  
and hence $\mathfrak{B}$ is nilpotent.
Then, by the definition of nilpotent blocks,
$\mathfrak B'$ and $b'$ are both nilpotent blocks
since $\mathfrak B$ and $b$ are nilpotent. 

Clearly $B\in\Bl(G\mid b)\cap\Bl(G\mid\mathfrak B)$
and $\mathfrak B\in\Bl(G[b]\mid b)$.
Hence, by \cite[Theorem 3.5]{DadeBlockExtensions} or \cite[Theorem 3.5(i)]{Murai_Dade}, $B$ is the unique block
of $G$ covering $\mathfrak B$.
In the following we denote by $1_C$ the block idempotent of the block $C$ over $k$. 
Moreover, \cite[Theorem 3.5]{DadeBlockExtensions} (or  \cite[Corollary 4]{Kuelshammer})
implies that $1_B = 1_{\mathfrak B}$.
Then, by the theorem of Harris-Kn{\"o}rr  \cite{HarrisKnoerr} (or \cite[Theorem (9.28)]{Navarro}), 
there exists a unique block
$B'\in\Bl(H\mid b')$ with $(B')^G = B$.
Analogously we see  $\mathfrak B' \in\Bl(L\mid b')$.
Note that the blocks $b$, $\mathfrak B$, $B$, $b'$, $\mathfrak B'$
and $B'$ have $D$ as defect group since $p\nmid |G/N|$,
see \cite[Theorem (9.26)]{Navarro}.

Note that according to \cite[Corollary 12.6]{DadeBlockExtensions} (or \cite[Theorem 3.13]{Murai_Dade})
we see that  $L = H[b']$. According to 
\cite[Theorem 3.5]{DadeBlockExtensions} (or   \cite[Theorem 3.5]{Murai_Dade}), $B'$ 
is the unique block of $H$ covering $\mathfrak B'$ and
hence the associated idempotents satisfy 
$1_{B'} = 1_{\mathfrak B'}$, just as above.
Since $\mathfrak B'$ is nilpotent,
$\IBr(\mathfrak B') := \{\psi\}$ for some $\psi\in\IBr(L)$.

Now, since ${\mathfrak B}{\downarrow}{_{N\times N}} \, \cong b$,
it follows from \cite[Theorem 8]{Kuelshammer} and
\cite[Theorem 4.1]{HidaKoshitani} that
${\mathfrak B'}{\downarrow}{_{M\times M}} \, \cong b'$.
Hence, by \cite[Theorem 4.1]{HidaKoshitani}, the map
$\IBr(\mathfrak B')\rightarrow\IBr(b')$ given by
$\theta\mapsto\theta_M$ is a bijection. This yields that
$ \psi_M = \phi'$.      

Since $\mathfrak B$ and $\mathfrak B'$ are nilpotent,
it holds from \cite[1.20.3]{KuelshammerPuig} that 
there exist an $\mathcal O$-algebra $\mathcal A$ of finite rank
and positive integers $n$ and $n'$ such that
\[ 
   B\cong\Mat_n(\mathcal A) \quad \text{and} \quad
   B'\cong\Mat_{n'}(\mathcal A) \quad
  \text{as }\mathcal O\text{-algebras}.
\]
Thus, the Morita equivalence between $B$ and $B'$ induces a bijection
$\Pi: \IBr(B)\rightarrow\IBr(B')$ such that
$(\Pi(\theta))(1) = \frac{n'}{n}\theta(1)$ for each $\theta\in\IBr(B)$.

Now, since $\IBr(\mathfrak B) = \{\w\phi_{G[b]}\}$,
$\w\phi(1) = \phi(1)$ and $D$ is a defect group of $\mathfrak B$,
it holds by a result of Puig \cite[(1.4.1)]{Puig1988} that
$\mathfrak B \cong \Mat_{\phi(1)}(\mathcal OD)$ 
as $\mathcal O$-algebras, and hence 
$\rank\mathfrak B = |D|\phi(1)^2$.
On the other hand, since $1_B = 1_{\mathfrak B}$, we have
$\rank B = |G:G[b]| \,\, \rank {\mathfrak B}$.
Similarly, since $\psi(1) = \phi'(1)$, it follows that
\[
\rank \mathfrak B' = |D|\phi'(1)^2
\quad \text{and} \quad 
\rank B' = |H:H[b']| \, \rank {\mathfrak B}.
\]
Note that $H[b'] = L = G[b]\cap H$, and 
$G = G_b = NH = G[b]H$ by an Frattini argument, hence
$|G:G[b]| = |H:H[b']|$. Thus,
\[ (\rank B)/n^2 = \rank \mathcal A
  = (\rank B')/(n')^2.
\]
Accordingly
\[
\Big(\frac{n'}{n}\Big)^2 = \frac{\rank B'}{\rank B}
= \frac{\rank\mathfrak B'{\cdot}|H:H[b']|}
       {\rank\mathfrak B{\cdot}|G:G[b]|}
= \frac{\rank\mathfrak B'}{\rank\mathfrak B}
= \frac{|D|\phi'(1)^2}{|D|\phi(1)^2}
= \Big(\frac{\phi'(1)}{\phi(1)}\Big)^2.
\]
This yields that $\psi(1) = \phi'(1) = \frac{n'}{n}\phi(1)$.
Let $\w\phi':=\Pi(\w\phi)\in \IBr(B')$. Since $\w\phi'(1) = \psi(1)$,
$B'$ covers $\mathfrak B'$ and 
$\IBr(\mathfrak B') = \{\psi\}$, we know that
$\w\phi'_L = \psi$ and hence $\w\phi'_M = \phi'$.
Thus $\phi'$ has an extension belonging to $B'$.

Recall that $\w\phi'$ was constructed such that
\[ \bl(\w\phi_{G[b]})= \bl(\psi)^{G[b]}=\bl(\w\phi'_{L})^{G[b]}.\]
Since $p\nmid |G:N|$ we can apply Lemma 2.4 of \cite{KoshitaniSpaeth1} 
and obtain that $\w\phi$ and $\w \phi'$ satisfy 
\[ \bl(\w\phi_{\spann<N,x>})= \bl(\w\phi'_{\spann<M,x>})^{{\spann<N,x>}}
\text{ for every } x\in L^0 .\]
Note that this implies for every $ x\in H^0$ using Theorem 3.5 of 
\cite{Murai_Dade} the following equalities 
\begin{align*}
 \bl(\w\phi_{\spann<N,x>})&=
\bl(\w\phi_{\spann<N,x>\cap G[b]})^{\spann<N,x>}= \\
&=
\bl(\w\phi'_{\spann<M,x>\cap L})^{{\spann<N,x>}}=
(\bl(\w\phi'_{\spann<M,x>\cap L})^{{\spann<M,x>}})^{{\spann<N,x>}}=
\bl(\w\phi'_{\spann<M,x>})^{{\spann<N,x>}}.\end{align*}
According to Lemma 2.5(a) of \cite{KoshitaniSpaeth1} this proves
\begin{align*}
 \bl(\w\phi_{J})&=
\bl(\w\phi'_{\NNN_J(D)})^{J}
\text{ for every } J \text{ with } N\leq J \leq G. \qedhere
\end{align*}
\end{proof}

We apply this to verify that the inductive BAW condition holds for nilpotent blocks of quasisimple groups. 
\renewcommand{\proofname}{Proof of Theorem \ref{thm1_3}}

\begin{proof}
Without loss of generality we may assume that the block $B$ does not
have central defect.
According to Theorem \ref{thm_An_Eaton} a defect group $D$ of $B$ 
is abelian. Hence for any $Q\in\Rad(G)$ the set $\dz(\NNN_G(Q)/Q,B)$ 
is non-empty if and only if $Q$ is a defect group of $B$, 
see proof of Consequence 2 in \cite{Alperin87}. 
Hence Definition \ref{def4_2}(i) is trivial.

Let $D$ be the defect group of $B$, and
let $B'\in\Bl(\NNN_G(D))$ be the Brauer correspondent of $B$.
Then, by the definition of nilpotent blocks,
$B'$ is also nilpotent.

The characters of $\dz(\NNN_G(D)/D,B)$ lift to characters in $\Irr(B')$. 
By the theory of nilpotent blocks
due to Brou{\'e} and Puig \cite{BrouePuigFrobenius},
see \cite[Theorem (52.8)]{Thevenaz}, 
there is exactly one character in $\Irr(B')$ that contains 
$D$ in its kernel. We identify $\dz(\NNN_G(D)/D,B)$ with 
$\IBr(B')$, since $\dz(\NNN_G(D)/D,B)$ correspond to characters 
in $\IBr(B')$. 

There exists an $\Aut(G)_{B}$-equivariant bijection 
$\Omega_D^G:\IBr(B)\longrightarrow \dz(\NNN_G(D)/D,B)$ 
since both sets contain exactly one character.
Obviously Definition \ref{def4_2}(iv) does not apply since $D\neq 1$.

Hence it suffices to check Condition \ref{def4_2}(iii)
Let $\phi\in \IBr(B)$. 
Set $Z := \Z(G) \cap \ker\phi$ and $\o G := G/Z$.
Note that $p\nmid |Z|$.
Clearly we can consider $\phi\in\IBr(\o G)$.
Let $\theta\in \IBr(\o G)$ be associated with $\phi$.
Hence there is a block $\o B\in\Bl(\o G)$ with $\bl(\theta) = \o B$.
By \cite[Theorem 5.8.8]{NagaoTsushima}, $B$ is the 
block of $G$ dominating $\o B$ and $\o D := DZ/Z \cong D$
is a defect group of $\o B$.
Hence, by the definition of nilpotent blocks,
we know that $\o B$ is also nilpotent, and hence
$\IBr(\o B) = \{\theta\}$.
Then, Lemma \ref{rem_exist_A} yields that there is a finite group $A$
such that $\o G \lhd A$,
$\Z(A) = \Cent_A(\o G)$, $A/Z(A)  \cong \Aut(G)_{\phi}$, $p \nmid |\Z(A)|$,
and that $\theta$ extends to $A$, namely there is
$\w{\theta}\in\IBr(A)$ with $\w{\theta}_{\o G} = \theta$.
Thus, Definition \ref{def4_2}(iii)(a)-(b) are satisfied.

Let $\theta'\in\IBr(\NNN_{\o G}(\o D))$ be the character associated with $\Omega_D^G(\phi)^0$. 
(This is well-defined since the lift of $\Omega_D^G(\phi)^0$ to $\NNN_G(D)$ 
covers the principal block of $Z$. Further by group theory $\NNN_{\o G}(\o D)= \NNN_G(D)/Z$ since $Z$ is a normal $p'$-subgroup.)
According to \cite[Theorem C(c)(2)]{KoshitaniSpaeth1}, 
$\theta'$ extends to $\NNN_{A[\o B]}(\o D)$ 
and some extension $\w{\theta'}\in\IBr(\NNN_{A[\o B]}(\o D))$ satisfies 
\[ \bl(\w{\theta'}_{\NNN_J(\o D)})^J=\bl(\w{\theta}_J) \text{ for every } 
\o G\leq J \leq A[\o B].\]

We check that we can apply Lemma \ref{lem4_2}. 
For every prime $q$ we denote by $A_q$ a group with 
$\o G\leq A_q\leq A$ and $A_q/\o G \in\Syl_q(A/\o G)$. 
For $p\neq q$ there exists an extension ${\psi}_q$ of $\theta'$ to 
$\NNN_{A_q}(\o D)$ according to Lemma \ref{prop3_7} such that 
\[ \bl((\psi_{q})_{\NNN_J(D)})^J=\bl(\ 
\theta_J)^J \text{ for every }\o G\leq J\leq A_q.\]
Hence the character 
$\psi_q$ satisfies 
$(\psi_q)_{\NNN_{A_q}(\o D)}= (\w{\theta'})_{\NNN_{A_q}(\o D)}$. 
Hence $\w{\theta'}$ extends to $\NNN_A(\o D)$ by Lemma \ref{lem4_2}. 
Accordingly Part (iii) of \ref{def_BAWC} is satisfied for $\phi$. 
\end{proof}
\renewcommand{\proofname}{Proof}

\section{Blocks with cyclic defect groups - Recall}\label{sec5}
In this section we recall some known results about 
blocks with cyclic defect groups that are relevant 
for our later considerations. 
Based on the work of Dade \cite{Dade66,Dade_cyclic}, 
blocks having cyclic defect groups seem well-understood. 
For Brauer characters Green correspondence gives a natural 
bijection with many additional properties.

We use for the Green correspondence the notation as introduced in 
\cite[Section 4.4]{NagaoTsushima}. 

Although Dade gave the following statement already in \cite{Dade_cyclic},  we recall for the sake of completeness its proof since its details are used later.

\begin{lem}\label{prop_bij_IBr}
Let $G$ be a finite group and $B\in\Bl(G)$ with a cyclic defect 
group $D$. Let $B'\in\Bl(\NNN_G(D))$ be the Brauer correspondent of $B$. 
Then there exists an $\Aut(G)_{B,D}$-equivariant bijection
\[ \Pi: \IBr(B) \longrightarrow \IBr(B'),\]
such that for $\phi \in \IBr(B)$ and a simple $kG$-module 
$V$ affording $\phi$, the character $\Pi (\phi)$ is the 
irreducible Brauer character of $\NNN_G(D)$ afforded by 
the head of $\f(V)$, where $\f = \f_{(G,D,\NNN_G(D))}$ is 
the Green correspondence with respect to $(G,D,\NNN_G(D))$.
\end{lem}

\begin{proof} The existence of a bijection can be deduced 
from Lemma 4.7 of \cite{Dade_cyclic} together with 
Theorem (9.9) of \cite{Navarro}. 
 Since any simple $kG$-module in $B$ has $D$ as its vertex, 
see  \cite{Dade66} or Corollary 3.7 in \cite{Knoerr}, 
we can define $\f(V)$. Then it is known that $\f(V)$ belongs 
to $B'$, where $B' \in \Bl(\NNN_G(D))$ is the Brauer correspondent 
of $B$, see Corollary 5.3.11 in \cite{NagaoTsushima}. Since any 
indecomposable $k\NNN_G(D)$-module in $B'$ is uniserial according 
to Section 19 in \cite{Alperinbook}, the head of $\f(V)$ is a 
simple $k\NNN_G(D)$-module in $B'$. Thus $\Pi$ is a well-defined 
map and bijective, see Chapter V in \cite{Alperinbook}. 
The bijection $\Pi$ is $\Aut(G)_{B,D}$-equivariant, 
since the 
Green correspondence has the analogous equivariance property.
\end{proof}

The theory of Dade from \cite{Dade_cyclic} and the Green 
correspondence provide several tools in this situation. 
In order to explore the bijection in Lemma \ref{prop_bij_IBr} 
we recall some facts about Green correspondence. For finite 
groups $H\leq G$ and an $H$-module $V$ we denote by $V^G$ the 
induced module, and for a $G$-module $W$ we denote by $W_H$ 
the restriction of $W$ to an $H$-module.

\begin{lem}\label{p'IndexGreen} 
Let $N\lhd G$ with $p \nmid |G : N|$. Further let $V$ be an 
indecomposable $kN$-module with vertex $D$ and suppose that 
there is a $kG$-module $\w V$ such that $\w V_N \cong V$ as 
$kN$-modules. For $H:= \NNN_G(D)$ and $M:= \NNN_N(D)$ the 
following holds:
\enumalph
\begin{enumerate}
    \item $D$ is a vertex of $\w V$.
    \item Let $\w\f$ and $\f$ be the Green correspondences 
with respect to $(G, D, H)$ and $(N, D, M)$, respectively. 
Then, 
$(\w\f \w V)_M$ is the direct sum of $\f V$ and indecomposable 
modules that do not have $D$ as their vertices. 
     \item $\w\f \w V$ is a direct summand of $(\f V)^H$.
\end{enumerate} 
\end{lem}
\begin{proof}
Clearly $\w V$ is indecomposable as a $kG$-module. Secondly, 
$D \subseteq_G {\vx}(\tilde V)$
by Lemma 4.3.4(ii) in \cite{NagaoTsushima}.
Since $\w V$ is relatively $N$-projective by
Theorem 4.2.5 in \cite{NagaoTsushima}, we know that
$\w V \mid ({\w V}_N)^G =  V^G$.
Thus, \cite[Lemma 4.3.4(i)]{NagaoTsushima} implies that
$D$ is a vertex of $\w V$.

Because of (a) we can apply Green correspondence and have 
\[{\w V}_M = ({\w V}_H)_M =\Big( \w\f \w V 
\oplus (\oplus_i \w Y_i) \Big)_M = (\w\f \w V)_M 
\oplus \Big( \oplus_i ((\w Y_i)_M) \Big),\]
where each $\w Y_i$ has a vertex which is in the 
$\w {\mathfrak Y}  := {\mathfrak Y} (G, D, H)$, 
that is defined as in Section 4.4.1 of \cite{NagaoTsushima}. 
We easily get by Lemma 4.3.4(ii) in \cite{NagaoTsushima} 
and the definition of $\mathfrak Y$ that any indecomposable 
direct summand of each $(\w Y_i)_M$ can not have $D$ as a vertex. 
Thus it follows that
\[
{\w V}_M = ({\w V}_H)_M = 
(\w\f \w V)_M \oplus \Big( \oplus \text{indec. $kM$-module
 with vertex $ \neq  D$}  \Big). \]

On the other hand,
\[\w V_M= ({\w V}_N)_M= V_M= \f V \oplus 
\Big( \oplus \text{indec. $kM$-module 
with vertex $ \neq  D$}  \Big).\]
Therefore, by Krull-Schmidt's theorem, we know the assertion of (b).

By the proof of (a) above it holds that $\w V \mid V^G$. 
Hence, Burry's Theorem in \cite[Theorem 4.4.8(ii)]{NagaoTsushima} 
implies that $\w\f \w V$ is a direct summand of $(\f V)^H$. This proves (c).
\end{proof}

The above statement on Green correspondence is applied 
in the situation of extending characters. 

\begin{lem} \label{ExtensionGreen}
Let $N \lhd G$ with $p \nmid |G:N|$. Further, suppose 
that $V$ is an indecomposable $kN$-module such that 
there is a $kG$-module $\w V$ 
with $\w V_N \cong V$ as $kN$-modules. Let $D$ be a vertex 
of $V$(and hence $\w V$ is an indecomposable $kG$-module 
with vertex $D$, see Lemma \ref{p'IndexGreen}(a)). 
For $H:= \NNN_G(D)$ and $M:= \NNN_N(D)$ it holds 
$$(\w\f \w V)_M \cong \f V $$ 
where $\w\f$ and $\f$ are the Green correspondences with 
respect to $(G, D, H)$ and  $(N, D, M)$, respectively.
\end{lem}

\begin{proof} First, recall that $\f V$ is $H$-invariant by 
definition. Now, since $M \lhd H$, Mackey formula in
\cite[Theorem 3.1.9]{NagaoTsushima} implies that
\begin{align*}
((\f V)^H)_M &=
\bigoplus_{h \in [M \backslash H / M]}
\Big( ( (\f V)^h) _{M^h \cap M}\Big)^M \\
&=\bigoplus_{h \in [H / M]} \Big( ((\f V)^h)_{M^h \cap M}\Big)^M 
=  \bigoplus_{h \in [H / M]}(\f V)^h,
\end{align*}
where for a $kM$-module $X$ we denote by $X^h$ the $kM$-module obtained via conjugation with $h\in H$.
We see that the last term that it is congruent to 
$|H/M|$-times copies of  $\f V$ as $kM$-module. Now, by Burry's theorem in
Lemma \ref{p'IndexGreen}(c) it follows that $\w\f\w V$ 
is a direct summand of the induced module $(\f V)^H$.
Hence, $(\w\f \w V)_M$ is a direct summand of $((\f V)^H)_M$, 
that is the direct sum of copies of $\f V$. Therefore it follows 
from Lemma \ref{p'IndexGreen}(b) that $(\w\f \w  V)_M$ is 
the direct sum of $\f V$ and indecomposable $kM$-modules with 
vertex different from  $D$. Thus, since $\f V$ has $D$ as a vertex, 
Krull-Schmidt's theorem implies that $(\w\f \w  V)_M = \f V$.
\end{proof}
 
We start with the following bijection between Brauer characters. 

\begin{lem}\label{prop4_2}
Let $N\lhd G$ and $b\in \Bl(N)$ with a cyclic defect 
group $D$, and let  $b'\in\Bl(\NNN_N(D))$ be the Brauer correspondent 
of $b$.
\enumalph \begin{enumerate}
\item Then there exists a natural $\NNN_G(D)_b$-equivariant bijection
\[ \Pi_{b,D}: \IBr(b)\longrightarrow \IBr(b').\] 
\item \label{prop4_2b}
Assume $p\nmid |G:N|$. Suppose that $B\in\Bl(G \mid  b)$ 
 and $\w\phi\in\IBr(B)$  with $\phi:=\w\phi_N\in\IBr(b)$. Then 
\[ \Big(\Pi_{B, D}(\w\phi)\Big)_{\NNN_N(D)} =\Pi_{b, D}(\phi).\]
\end{enumerate}
\end{lem}

\begin{proof} Part (a) follows from Lemma \ref{prop_bij_IBr}. 
Part (b) is a consequence of Lemma \ref{ExtensionGreen} 
and the definition of $\Pi_{b,D}$ using the Green correspondence.
\end{proof}

In the later we use the notation of ordinary characters in those blocks.

\begin{notation}[Characters in blocks with cyclic defect groups]\label{4_5}
Let $G$ be a finite group and $B\in\Bl(G)$ a block with 
cyclic defect group $D$. Let $e$ be the inertial index of $B$ 
and $\Lambda$ be a representative set of the 
$\NNN_G(D,b_0)$-orbits on $\Irr(D)\setminus\{1_D\}$,
where $b_0\in\Bl(\Cent_G(D))$ with $(b_0)^G = B$.
We denote by 
$\chi_1,\ldots,\chi_e,\{ \chi_\la \mid \la\in\Lambda\}$ 
the ordinary characters of $B$ as in \cite[\S 68]{Dornhoff}. 
We write  $\Irr_{nex}(B)$ for the set of non-exceptional 
characters of $B$ and $\Irr_{ex}(B)$ for the  set 
$\{ \chi_\la \mid \la\in\Lambda\}$. We denote by 
$\phi_1,\ldots,\phi_e$ the irreducible Brauer characters  
of $B$. \end{notation}
The exceptional characters can be described by using the 
$*$-construction from  \cite{BrouePuig80}.

\subsection{The Brou\'e-Puig $*$-construction of class functions} 
Let $G$ be a finite group and $B$ a $p$-block of $G$ with 
a maximal $B$-Brauer pair $(D, b_D)$ (and hence $D$ is 
a defect group of $B$, and $b_D$ is a $p$-block of $C_G(D)$ 
with $(b_D)^G = B$). Let $\chi$ be any $\mathcal O$-valued 
class function defined on $G$ and let $\nu$ be any 
$\mathcal O$-valued class function defined on $D$ such that if 
$x \in G$ and if $(u,f)$ is a Brauer element with 
$(u,f) \in (D,b_D)$ and $(u,f)^x \in (D, b_D)$ then 
$\nu(u) = \nu(u^x)$, where $u^x := x^{-1}ux$. Then, 
the $*$-construction $\chi * \nu$ is well-defined, 
and actually $\chi * \nu$ is a generalized character 
of $G$ in $B$, see Theorem of \cite{BrouePuig80}. 
Note that this definition depends on the choice of 
$(D, b_D)$, see Remark 1 on p.446 of \cite{Cabanes88}.

Let $B$ be a block with cyclic defect group $D$, $p\neq 2$, $\la\in\La$ and $\eta_\la$
the sum of $\NNN_G(D)$-conjugates of $\la$. Then straight-forward 
calculations, as in Lemma 8 of  \cite{Watanabe12}, prove
that the exceptional character 
$\chi_\la$ is the unique constituent of $\chi_1*\eta_\la$ 
that is not $p$-rational.
For $p=2$ the characters $\chi_\la$ and $\chi_1*\la$ coincide.

In our later considerations we use the following 
well-known facts about blocks with cyclic defect groups

\begin{lem} Let $B$ be a nilpotent $p$-block with cyclic non-trivial defect group $D$. Let $\chi_1\in\Irr(B)$ be defined as in Notation \ref{4_5}.
\enumalph \begin{enumerate}
\item If $p$ is odd, then the non-exceptional characters of 
$B$ are the $p$-rational characters in $\Irr(B)$. 
\item \label{rem5_7a} For $p=2$ any block with cyclic defect is nilpotent.
\item \label{rem5_7c} For $p=2$ there are exactly two $2$-rational ordinary irreducible characters of $B$, namely $\chi_1$ and $\chi_1*\delta$ where $\delta$ is the unique character of $D$ of order $2$. 
\end{enumerate}
\end{lem}

\begin{proof}
According to Theorem 68.1(8) of \cite{Dornhoff} part (a) holds. 
Part (b) follows from the fact that the automorphism group of $D$ 
is then a $2$-group and hence $b$ is nilpotent. 
Part (c) is the footnote on p. 26 of \cite{Dade66} and follows from 
Part 3 of Theorem 1 of \cite{Dade66}.
\end{proof}
In order to deduce later from the considerations on the inductive 
AM condition for those blocks and the inductive BAW condition 
it is useful to know the following property of the 
decomposition matrix. Although it seems to be a well-known 
fact for the sake of completeness we give here a proof. 

\begin{thm}\label{dec_matrix_unitriangular}
Let $B\in\Bl(G)$ be a block with cyclic defect group $D$. Let $e$ be the 
inertial index of $B$ and $\Lambda$ a representative set 
of the $\NNN_G(D,b_0)$-orbits on $\Irr(D)\setminus\{1_D\}$. 
We denote by $\chi_1,\ldots,\chi_e,\{ \chi_\la \mid \la\in\Lambda\}$ 
the ordinary characters of $B$ as in \cite[\S 68]{Dornhoff}. 
Then we can label the non-exceptional characters 
$\chi_1,\ldots, \chi_e$ and the irreducible Brauer 
characters $\phi_1,\ldots,\phi_e$ of $B$ such that 
the associated decomposition matrix is unitriangular, i.e.
\[ (\chi_i)^0 =\sum_{j=1}^i d_{i,j} \phi_j\]
for some non-negative integers $d_{i,j}$ with $d_{i,i}=1$.
\end{thm}

\begin{proof}
The proof uses intensively the results of \cite{Dade66}, 
see also \cite[\S 68]{Dornhoff}. First recall that according to 
\cite[Theorem 68.1]{Dornhoff} the decomposition numbers 
are contained in $\{0,1\}$. Hence the decomposition 
matrix of $B$ is determined by the number 
of exceptional characters and the Brauer graph of 
$B$, that is a tree according to \cite[Corollary 68.2]{Dornhoff}. 
Note that the Brauer tree has exactly one vertex associated 
with the exceptional characters called the exceptional vertex. 
We label the exceptional vertex by $e+1$ and the other 
vertices of the Brauer tree can be labeled by $\{1,\ldots, e\}$ 
and the edges by $\{1,\ldots, e\}$, such that the edge 
labeled by $i$ connects the vertex labeled by $i$ and 
some vertex labeled by $i'$ with $i'>i$.  Let $1\leq i \leq e$ 
and $I$ the labels of the edges connected with the vertex 
corresponding to $i$. Then $(\chi_i)^{0}$ coincides with 
$\sum_{j\in I}\phi_j$. By the choice of the labeling 
we have that $I\subseteq\{1,\ldots, i\}$. This proves the 
statement. \end{proof}

\section{Beyond blocks with cyclic defect groups}\label{sec5_beyond}
In this section we apply the recalled results from the previous section to 
describe the Clifford theory of characters belonging to blocks with cyclic defect groups. 

As a result of this section we see that various properties of characters belonging to a 
block with cyclic defect groups are already determined locally. 
Using Notation \ref{4_5} we establish equivariant bijections between the characters of Brauer corresponding 
blocks with cyclic defect groups. Further we study the existence of certain extensions. 
Since we are using various rationality arguments we treat the case where $p$ is odd
and where $p$ is even separately. 

For odd primes $p$ there is a natural bijection between 
characters of a block with cyclic defect groups and 
those of its Brauer correspondent. 

\begin{prop}\label{prop4_9}
Let $p$ be an odd prime, $N$ a finite group, and 
$b\in\Bl(N)$ with cyclic defect group $D$. 
Then there exists a unique $\Aut(N)_D$-equivariant bijection 
\[ \La_{b,D}: \Irr(b)\longrightarrow\Irr(b'),\]
where $b'\in\Bl(\NNN_N(D))$ is the Brauer correspondent of $b$. 
Further for every $\nu\in\Irr(\Z(N))$ the bijection satisfies 
the following inclusion
$\La_{b,D}(\Irr(b)\cap \Irr(N\mid \nu))\subseteq \Irr(\NNN_N(D)\mid \nu)$.
\end{prop}

The bijection is a consequence of Dade's work in 
\cite{Dade66} and \cite{Dade_cyclic}, see also \cite[\S 68]{Dornhoff}. 
For later applications we present a detailed construction of the bijection. 
\begin{proof}
As mentioned above $\Irr(b)=\Irr_{nex}(b) \cup \Irr_{ex}(b)$. 
Since $\Irr_{nex}(b)$ coincides with the set of 
$p$-rational characters in $\Irr(b)$, hence $\Irr_{nex}(b)$ 
and $\Irr_{ex}(b)$ are $\Aut(N)_b$-stable. 
 
Note that any $p$-rational character of $N$ is trivial on 
the Sylow $p$-subgroup of $\Z(N)$.                               
Let $a$ be the integer with $|D|=p^a$, and $D_i$ the subgroup 
of $D$ with $|D:D_i|=p^i$ for each integer $i$. We assume that $D_{a-1}$ is non-central. 
Otherwise the characters are considered as characters of $G/D_{a-1}$. 
According to Lemma 4.10 of \cite{Dade_cyclic} there exists a 
bijection $\Irr_{nex}(b)$ to $\Irr_{nex}(c_{a-1})$ where 
$c_i\in\Bl(\NNN_N(D_{i}))$ is the block with $(c_i)^N=b$. 
The characters of $\Irr_{nex}(c_{a-1})$ can be identified 
with the characters of the block 
$\o c_{a-1}\in\Bl(\NNN_N(D_{a-1})/{D_{a-1}})$, 
that is dominated in $c_{a-1}$. By Lemma (3.3) of 
\cite{NavarroMcKay} the block $\o c_{a-1}$ has defect 
group $D/D_{a-1}$. Successively applying this procedure 
we obtain an $\Aut(N)_b$-equivariant bijection 
\[ \Lambda_{b,D,nex}:\Irr_{nex}(b) \longrightarrow\Irr_{nex}(b').\]
We can choose a labeling of $\Irr_{nex}(b')$, such that 
$\Lambda_{b,D,nex} (\chi_i)=\chi'_i$ for every $1\leq i \leq e$. 
Let $\Z(N)_{p'}$ be the Hall $p'$-subgroup of $\Z(N)$. 
Since $b$ and $b'$ cover the same $p$-block of 
$\Z(N)_{p'}$, the bijection satisfies
\[\La_{b,D,nex}(\Irr(b)\cap \Irr(N\mid \nu))
\subseteq \Irr(\NNN_N(D)\mid \nu).\]

The exceptional characters of $b'$ are denoted 
by $\chi'_\la $ for $1\neq\la\in\Irr(D)$. Furthermore 
we define $\La_{b,D}:\Irr(b)\longrightarrow \Irr(b')$ 
to be the map that coincides with $\La_{b,D,nex}$ on 
$\Irr_{nex}(b)$ and satisfies  $\La_{b,D}(\chi_\la)=\chi'_\la 
$ for every $\la\in\La$. 

Let $\Z(N)_p$ be a Sylow $p$-subgroup of $\Z(N)$.  
The characters $\chi_\la$ and $\chi'_\la$ satisfy the formulas from 
\cite[p.1135]{NavarroMcKay}. Accordingly for 
$\nu\in\Irr(\Z(N)_{p})$ with $\la\in\Irr(D\mid \nu)$ 
the characters $\chi_\la$ and $\chi_{\la'}$ are contained 
in $\Irr(N\mid \nu)$ and $\Irr(\NNN_N(D)\mid \nu)$, respectively. 

For the existence of the bijection it remains to prove that $\La_{b,D}$ is $\Aut(N)_{D,b}$-equivariant  on  $\Irr_{ex}(b)$. 
Let $\phi\in\IBr(\Cent_N(D))$ such that $\bl(\phi)^N=b$. 
Then $\Aut(N)_{D,b}$ is generated by automorphisms induced 
by $\NNN_N(D)$ and $\Aut(N)_{D,\phi}$. After choosing 
$\phi$ the character $\chi_1* \eta_\la$ is uniquely defined 
where $\eta_\la$ is the sum of characters that are 
$\NNN_N(D)$-conjugate to $\la$. For $\la\neq 1$ the 
character $\chi_\la$ is a unique non-$p$-rational 
constituent of $\chi_1* \eta_\la$. Hence the stabilizer 
of $\chi_\la$ in $\Aut(N)_{D,\phi}$ coincides with 
$\Aut(N)_{D,\phi,\la}$. Further for 
$\sigma\in\Aut(N)_{D,\phi}$ and $1\neq\la\in\Irr(D)$ the character $(\chi_\la)^\sigma$ 
coincides with $\chi_{\la^\sigma}$ since $\chi_\la$ is the unique non-$p$-rational constituent of 
$\chi*\la$.  An analogous statement holds for 
$\chi'_\la$ and hence the map $\La_{b,D}$ is an 
$\Aut(N)_{b,D}$-equivariant bijection. This proves the statement.
\end{proof}

This bijection is compatible with "going-to-quotients". 

\begin{cor} \label{cor5_2}
 Let $p$ be an odd prime, $N$ a finite group, and $Z\leq \Z(N)$. Let 
$b\in\Bl(N)$ be a $p$-block with cyclic defect group $D$. 
Let $\o b$ be the unique block of $N/Z$ dominated by $b$ and 
$\o \chi\in\Irr(\o b)$. Let $\chi\in\Irr(b)$ be the lift 
of $\o \chi$. Then $\La_{\o b, DZ/Z}(\o\chi)$ 
defined as in Proposition \ref{prop4_9} lifts to $\La_{b,D}(\chi)$.
\end{cor}

\begin{proof}
If $\chi\in\Irr_{nex}(b)$, then the character $\o \chi$ is 
non-exceptional, as well. Then $\La_{\o b, DZ/Z}(\o\chi)$ 
and $\La_{ b, D}(\chi)$ are defined using the Green 
correspondence. By the definition of $\La_{b,D}$ and 
$\La_{\o b, DZ/Z}$ the statement also holds for exceptional characters. 
\end{proof}

The bijection $\La_{b,D}$ also preserves the 
"property of extendibility of characters", more 
precisely it maps characters of $b$ that extend to $G$, 
to characters of $b'$ with a similar property. 

\begin{prop}\label{prop5_11}
Let $p$ be an odd prime, $N\lhd G$ with $p\nmid |G:N|$, and 
$b\in\Bl(N)$ a $p$-block with a cyclic non-central 
defect group $D$. Let $B\in\Bl(G\mid b)$ 
and $\w\chi\in\Irr(B)$ a character such that 
$\chi:=\w\chi_N\in\Irr(N)$. Then there is an extension 
$\w\chi'$ of $\La_{b,D}(\chi)$ to $\NNN_G(D)$ such that 
$\bl(\w\chi')^G=\bl(\w\chi)$.
\end{prop} 

\begin{proof} 
Let the characters in $\Irr_{nex}(b)\cup\IBr(b)$ be labeled 
as in Theorem \ref{dec_matrix_unitriangular}. 
First assume that $\chi$ is non-exceptional 
and hence $\chi=\chi_i$ for some $i$. 
According to Theorem \ref{dec_matrix_unitriangular} 
the character $\phi_i$ is invariant in $G$ and is 
a constituent of $\chi^0$ with multiplicity $1$. 
Hence $\phi_i$ extends to $G$. 

Let $b'\in\IBr(\NNN_N(D))$ be the Brauer correspondent of $b$.
From Lemma \ref{prop4_2b} we know that some character of 
$\IBr(b')$ extends to $\NNN_G(D)$. 
Let $B'$ be a block of $\NNN_G(D)$ to which        
this extension belongs.
Note that $B'$ has $D$ as a defect group according to 
\cite[Theorem (9.26)]{Navarro}.                                
Since $D$ is cyclic, all characters of 
$\IBr(B')$ have the same degree. 
This implies that every character of $\IBr(b')$ extends to  
$\NNN_G(D)$. Since all ordinary 
non-exceptional characters are lifts of those characters, 
$\La_{b,D}(\chi)$ extends to $\NNN_G(D)$, as well. 

Next we consider the case where $\chi\in\Irr_{ex}(b)$. 
Since $\La_{b,D}$ is $\NNN_G(D)$-equivariant,  
$\chi':=\La_{b,D}(\chi)$ is non-exceptional and 
$\NNN_G(D)$-invariant. 
Let $\la\in\Irr(D)$ be such 
that $\chi'=\chi'_\la$ in the notation of \ref{4_5}.
This implies 
\[\NNN_G(D)=\NNN_N(D)\NNN_G(D)_{\la}.\]
Straight-forward considerations using the structure of the groups show that
$\NNN_G(D)_\la= \Cent_G(D)_\la$ since $\NNN_G(D)_\la/\Cent_N(D)$
is a $p'$-group.

Let $\phi\in\IBr(\Cent_G(D))$ such that $\bl(\phi)$ is 
covered by $b'$, and $\w\phi$ some extension of 
$\phi$ to $(\NNN_N(D))_\phi$. Then we denote by 
$C_\phi$ and $C_{\w\phi}$ their stabilizers in 
$C:=\Cent_G(D)$. As in \cite[Lemma 3.12]{Murai_Dade} 
we see that $C_{\w\phi}$ acts on the set of extensions of 
$\phi$ to $\NNN_N(D)_\phi$ by multiplication with a 
linear character of $\NNN_N(D)_\phi/\Cent_N(D)$. Note 
that $\NNN_N(D)_\phi/\Cent_N(D)$ is a cyclic $p'$-group, 
and hence $C_\phi/C_{\w\phi}$ is cyclic as well. From 
\cite[Lemma 3.12 and Theorem 3.13]{Murai_Dade} we know that 
\[ \NNN_G(D) [b']= \NNN_N(D) C_{\w \phi},\]
since $C_{\w\phi}$ is the stabilizer of a bilinear 
form defined in the subsection before Lemma 3.12 of \cite{Murai_Dade},
see also \cite[Corollary 12.6]{DadeBlockExtensions} 
for the original proof. Since $\chi'$ is an exceptional 
character, $\chi':=\psi^N$ for some 
$\psi\in\Irr(\Cent_G(D)\mid \la)$, where 
$1\neq \la\in\Irr(D)$ and $\bl(\psi)=\bl(\phi)$. 
Note that $\psi$ is the unique character in 
$\Irr(\Cent_G(D)\mid  \la)$ with $\psi^0=\phi$.  
Since $\chi'$ is $\NNN_G(D)$-invariant, 
\begin{align*}
\NNN_G(D)&=(\NNN_G(D))_{\chi'} 
= \NNN_N(D) (\NNN_G(D)_\psi)=\\                            
&= \NNN_N(D) (\NNN_G(D)_ \la\cap \NNN_G(D)_\phi)
=\NNN_N(D) \Cent_G(D)_{\phi}.
\end{align*}
According to 
\cite[Theorem C(a)(2)]{KoshitaniSpaeth1} 
combined with \cite[Theorem 4.1]{Murai_Dade}, 
$\chi'$ has a unique extension $\w\chi'$ to 
$\NNN_G(D)[b']$ such that 
\[\bl(\w\chi')^{G[b]}=\bl(\w\chi_{G[b]}).\]
According to \cite[Proposition 9]{Kuelshammer} we 
have $N\NNN_G(D)[b']=G[b]$ accordingly there is a 
unique block $\w b' \in\Bl(\NNN_G(D)[b']\mid b')$ with 
$(\w b')^{G[b]}=\bl(\w\chi_{G[b]})$, see \cite[Theorem (9.28)]{Navarro}
(or \cite{HarrisKnoerr}). The block 
$\bl( \w\chi_{G[b]})$ is $G$-invariant by definition. 
Hence its Harris-Kn\"orr correspondent $\w b'$ is 
$\NNN_G(D)$-invariant, as well. By the definition 
of $\w\chi'$ this implies that $\w\chi'$ is $\NNN_G(D)$-invariant.  
Since $\NNN_G(D)/\NNN_G(D)[b']$ is isomorphic to 
$C_\phi/C_{\w \phi}$ and hence cyclic, $\w\chi'$ has 
an extension $\wh \chi'$ to $\NNN_G(D)$, 
see \cite[Corollary (11.22)]{Isa}.               
According to \cite[Lemma 3.3 and Proposition 1.9]{DadeBlockExtensions}
(see also \cite[Theorem 3.5(i)]{Murai_Dade}),
this extension satisfies
\begin{align*}
 \bl(\wh \chi')^G&=\bl(\w\chi).\qedhere
\end{align*}
\end{proof}

A statement like above holds when $G/N$ is a $p$-group.

\begin{lem}\label{prop5_12}
Let $p$ be an odd prime, $N\lhd G$ such that $G/N$ is a $p$-group,
and $b\in\Bl(N)$ $p$-block with cyclic defect group $D$. 
Then, a character $\chi\in\Irr(b)$ extends to $G$ 
if and only if $\La_{b,D}(\chi)$ extends to $\NNN_G(D)$. 
Furthermore the extensions $\w\chi\in\Irr(G)$ of $\chi$ 
and $\w\chi'\in\Irr(\NNN_G(D))$  of 
$\La_{b,D}(\chi)$ can be chosen such that 
\[ \Irr(\Z(G)\mid \w\chi)=\Irr(\Z(G)\mid \w\chi').\]
\end{lem}
\begin{proof}
We consider first the case where $\chi$ is a non-exceptional 
character. Then $\chi$ is $p$-rational. Analogously by the definition 
of $\La_{b,D}$ the character $\chi':=\La_{b,D}(\chi)$ is 
non-exceptional and hence $p$-rational, as well. Since 
$\chi'$ is $\NNN_G(D)$-invariant, $\chi'$ extends to 
some $p$-rational character of $\NNN_G(D)$ according to 
\cite[Theorem (6.30)]{Isa}. In this situation let $\w\chi$ and 
$\w\chi'$  be $p$-rational extensions of $\chi$ and $\chi'$. Accordingly 
$\Irr(\Z(G)\mid \w\chi)$ and $\Irr(\Z(G)\mid \w\chi')$ 
contain only $p$-rational characters. This implies 
that both $\ker(\w\chi)$ and $\ker(\w\chi')$ contain the Sylow 
$p$-subgroup of $\Z(G)$. Further for every $p'$-group 
$Z\leq \Z(G)$ the blocks $\bl(\chi)$ and $\bl(\chi')$ 
cover the same block of $Z$. 
This proves 
\[\Irr(\Z(G)\mid \w\chi)=\Irr(\Z(G)\mid \w\chi').\]

Now we assume that $\chi\in\Irr_{ex}(b)$ and that 
$b$ is not nilpotent. 
Let $b'\in\Bl(\NNN_N(D))$ be the Brauer correspondent of $b$.
Then $\chi':=\La_{b,D}(\chi)\in\Irr_{ex}(b')$ extends 
to its stabilizer in $\NNN_G(D)_{\chi'}$, see paragraph before 
Lemma 10 in \cite{Watanabe12}. The formula defining the 
extension of $\chi$ there allows that one can choose an 
extension of $\chi'$ with the required property. 

If $b$ is nilpotent, then $b'$ is also nilpotent. Let 
$\chi_1\in\Irr_{ex}(b')$. 
The character $\chi_1$ is $\NNN_G(D)$-invariant, 
since $b'$ is $\NNN_G(D)$-stable and $\chi_1$ is a unique 
$p$-rational character. This implies that $\chi_1$ extends 
to some $\w\chi_1\in\Irr(G)$ and $\w\chi=\chi_1* \nu$ for some 
character $\nu\in\Irr(\w D)$ where $\w D$ is a defect group of 
$\bl(\w \chi)$ containing $D$. Further $\chi=\chi_1*\nu_1$ where 
$\nu_1=\nu_D$. Hence by the above $\chi'_1\in\Irr(b')$ extends 
to some $\w\chi'_1\in\Irr(\NNN_G(D))$.  According to 
\cite[Theorem 1(ii)]{Cabanes87} the character $\w\chi'_1* \nu$ is an extension 
of $\chi':=\chi_1'* \nu_1$. This extension of $\chi'$ has the 
required property. 
\end{proof}

For $p=2$ similar results hold but their proofs become 
more involved since the block contains exactly two $2$-rational ordinary irreducible characters. 
On the other hand those blocks have the advantage to be nilpotent. 
We use the following well-known fact about nilpotent blocks.

\begin{lem}\label{rem5_14}
Let $B\in\Bl(G)$ be a nilpotent block with defect group 
$D$ and $Z$ a central $p$-subgroup of $G$. 
Then there is a unique block 
of $G/Z$ dominated by $B$ and this block is nilpotent.
\end{lem}

\begin{proof}
The block $B$ dominates a unique block of $G/Z$, see
Theorem 5.8.11 of \cite{NagaoTsushima}. Following the 
definition of nilpotent blocks using Brauer pairs we see 
that this block is nilpotent as well. 
\end{proof}
This is used to prove the following statement.
\begin{prop}\label{prop5_13}
Let $b\in\Bl(N)$ be a $2$-block with non-central cyclic defect group 
$D$, and let $b'\in\Bl(\NNN_N(D))$ be the Brauer correspondent of $b$. 
Associated to the two $2$-rational characters in $\Irr(b)$
 there exist two $\Aut(N)_{b,D}$-equivariant 
bijections 
\[ \La_{b,D}: \Irr(b)\longrightarrow\Irr(b') \text{ and } 
 \La'_{b,D}: \Irr(b)\longrightarrow\Irr(b').\] 
with 
\[ \La_{b,D}(\Irr(b)\cap \Irr(N\mid \mu))  \subseteq 
\Irr(\NNN_N(D)\mid \mu)\]
and 
\[\La'_{b,D}(\Irr(b)\cap \Irr(N\mid \mu)) \subseteq 
\Irr(\NNN_N(D)\mid \mu)\text{ for every }\mu\in\Irr(\Z(N)).\]
Note that in particular the two $2$-rational ordinary irreducible characters are 
$\Aut(N)_{b,D}$-invariant.
\end{prop}

\begin{proof}
Any $2$-block with cyclic defect group is nilpotent 
according to Lemma \ref{rem5_7a}. 
Hence according to \cite[Theorem (52.8)]{Thevenaz} there exists 
two $2$-rational characters in $\Irr(b)$, namely
$\chi_1$ and $\chi_1*\delta$, where $\delta$ is the character 
of $D$ of order $2$, see Lemma \ref{rem5_7c}. Further let $\chi_1'\in\Irr(b')$ be the 
canonical character of $b'$, i.e., the one with $D\leq \ker(\chi_1')$. 
Then $\La_{b,D}:\Irr(b)\longrightarrow \Irr(b')$ given by 
$\chi_1 *\nu \longmapsto \chi_1'*\nu$ for $\nu\in\Irr(D)$ is a
bijection. 

Now we consider the action of $\Aut(N)_{b,D}$ on both sets. 
Note that every $\sigma\in\Aut(N)_{b,D}$ fixes $\chi'_1$, 
and accordingly $(\chi_1*\nu)^\sigma=\chi_1^{\sigma}*(\nu^\sigma)$ according to Remark 1 of \cite{Cabanes88}. 
Since $(\chi_1)^\sigma$ is $2$-rational,  
$(\chi_1)^\sigma\in\{\chi_1,\chi_1* \delta\}$. Let 
$\sigma\in\Aut(N)_{b,D}$ with $\chi_1^\sigma\neq\chi_1$. 
Without loss of generality we can assume that $\sigma$ is a 
$2$-element. The group $\w N:=N\rtimes \langle \sigma\rangle$ 
has a unique block $\w b$ covering $b$. This block is 
nilpotent according to Theorem 2 of \cite{Cabanes87}. This implies that 
$\w b$ has a $2$-rational ordinary irreducible character $\psi$ of height zero. 
Since $\psi_N\in\Irr(b)$ we have 
$\psi_N\in\{\chi_1,\chi_1* \delta\}$. Further 
$\psi_N$ is $\sigma$-invariant. This implies that 
$\chi_1$ and $\chi_1*\delta$ are $\Aut(N)_{b,D}$-invariant. 

Accordingly the maps
\begin{align*}
\La_{b,D}: \Irr(b)\longrightarrow \Irr(b')&\text{ with }  
\chi_1*\nu \longmapsto \chi_1'*\nu \text { and }\\
\La'_{b,D}: \Irr(b)\longrightarrow \Irr(b')& \text{ with }  
\chi_1*\nu \longmapsto \chi_1'*(\nu\delta)\end{align*}
with $\nu\in\Irr(D)$ are 
$\Aut(N)_{b,D}$-equivariant by the definition of the $*$-construction.

It remains to verify the inclusion 
\[ \La_{b,D}(\Irr(b)\cap \Irr(N\mid \mu))\subseteq 
\Irr(\NNN_N(D)\mid \mu)\]
for every $\mu\in\Irr(\Z(N))$. According to  
considerations on the covered blocks of Hall 
$p'$-group of $\Z(N)$ in the proof of Proposition
\ref{prop4_9} it is sufficient to prove 
\[ \La_{b,D}(\Irr(b)\cap \Irr(N\mid \mu))
\subseteq \Irr(\NNN_N(D)\mid \mu)\]
for every $\mu\in\Irr(Z)$ where $Z$ is the 
Sylow $2$-subgroup of $\Z(N)$. 
Since $D$ is non-central, $Z\neq D$ and hence $Z\leq \ker(\delta)$.
Since $\o b\in \Bl(N/Z)$ 
the block contained in $b$ is nilpotent by Lemma
\ref{rem5_14}, and has $D/Z$ as defect group.
Since $D/Z\neq 1$ is non-trivial the set
$\Irr(\o b)$ contains two $2$-rational characters and those lift to  $\chi_1$ and $\chi_1*\delta$.
This proves $Z\leq\ker(\chi_1)$ and $Z\leq \ker(\chi_1*\delta)$.

For every $\nu\in\Irr(D)$ the set  $\Irr(N\mid \nu_Z)$ 
contains $\chi_1*\nu$ and $\chi_1*(\nu\delta)$. Since an 
analogous statement holds for exceptional characters of 
$b'$ this proves the claimed inclusions. 
 \end{proof}

\begin{cor}\label{cor5_7}
Let $N\lhd G$ and $b\in\Bl(N)$ a 
$p$-block with non-normal cyclic defect group $D$. 
Let $b'\in\Bl(\NNN_N(D))$ be the Brauer correspondent 
of $b$ and  $Z\leq \Z(N)$ such that some character of 
$\Irr(b)$ contains $Z$ in its kernel.
Then the bijections $\La_{b,D}$ and $\La'_{b,D}$ 
from Proposition \ref{prop5_13} induce the two analogously defined 
bijections  $\La_{\o b, \o D}$ and $\La'_{\o b, \o D}$ where $\o b$ 
is the unique block of $N/Z$ dominated by  $b$. 
\end{cor}

In the following we prove an analogue of Proposition \ref{prop5_11} for $p=2$. 

\begin{prop}\label{prop5_8}
Let $N\lhd G$ with $2\nmid |G:N|$, $\w\chi\in\Irr(G)$ and $b\in\Bl(N)$ be a $2$-block with cyclic defect group $D$. Further
let $b'\in\Bl(\NNN_N(D))$ be the Brauer correspondent of $b$.  
If $\chi:=\w\chi_N\in\Irr(b)$,  
then every $\chi'\in\Irr(b)$ extends to some $\w\chi'\in\Irr(\NNN_G(D))$ with $\bl(\w\chi')^G=\bl(\w\chi)$.  
\end{prop} 

\begin{proof}
Since $\chi$ and $b$ are $G$-invariant $G=N\NNN_G(D)$. Note that  
$\NNN_G(D)/\Cent_G(D)$ is a subgroup of the $2$-group $\Aut(D)$, 
and hence $G=N\Cent_G(D)$. 
Analogously we see that $G=N \Cent_G( D)_{b'}$. 
Let $\phi_0\in\IBr(\Cent_N(D))$ such that $\bl(\phi_0)^{\NNN_N(D)}=b'$.

First we show that some character of $b'$ has an extension in $B'\in\Bl(\NNN_G(D))$
the Harris-Kn\"orr corresponding block of $B$.

Lemma 6.2 of \cite{Spaeth_AM_red} shows that $B:=\bl(\w\chi)$ 
is nilpotent. Then $B'$ is nilpotent as well. 
Since $\chi$ extends to $G$, 
the corresponding Brauer character $\chi^0$ is irreducible and 
extends to $G$. By Lemma \ref{prop4_2b} 
the Brauer character $\phi$ of $b'$ extends to some $\w\phi\in\IBr(B')$ as well. 
As $B'$ is nilpotent, $\w\phi=\psi^0$ for some $\psi\in\Irr(B')$ and $\psi_{\NNN_N(D)}\in\Irr(b')$.

In the next step we prove $\NNN_G(D)=\NNN_G(D)[b']$. 
Let $\omega$ be the bilinear form on 
$\NNN_N(D)_{\phi_0}\times \Cent_G(D)_{\phi_0}$ defined in Section 3 
of \cite{Murai_Dade}. 
This bilinear form is invariant in 
$\Cent_G(D)_{\phi_0}$ since $\NNN_N(D)_{\phi_0}=\Cent_N(D)$. 
Now by \cite[Corollary 12.6]{DadeBlockExtensions} (see also \cite[Theorem 3.13]{Murai_Dade})
we obtain
\begin{align*}
\NNN_G(D)=\NNN_N(D) \Cent_G(D)_{\phi_0}=\NNN_N(D) \Cent_G(D)_\omega
= \NNN_G(D) [b'].
\end{align*}
This implies $G[b]=G$ as well. 
We can apply Theorem 4.1(ii) of \cite{Murai_Dade}. 
It implies that every  $\chi'\in\Irr(b')$ 
extends to some character in $B'$, 
and all irreducible ordinary characters of $b$ have an extension in $B$.
\end{proof} 

Only a light version of the above statement holds when $G/N$ is 
a $2$-group.

\begin{prop}\label{prop5_17}
Let $N\lhd G$, $\w\chi\in\Irr(G)$ and let $b\in\Bl(N)$ be a $2$-block 
with cyclic defect group $D$. 
Assume that $G/N$ is a $2$-group. 
Let $\La_{b,D}$ and $\La_{b,D}'$ be the bijections from
Proposition \ref{prop5_13}. If $\chi:=\w\chi_N\in\Irr(b)$  
is not $2$-rational, then 
$\La_{b,D}(\chi)$ and $\La_{b,D}'(\chi)$ 
extend to $\NNN_G(D)$.
\end{prop}

\begin{proof} 
Since $G/N$ is a $2$-group, the block $B :=\bl(\w \chi)$ 
is nilpotent by Theorem 2 of \cite{Cabanes87}. Let $\kappa$ be a $2$-rational 
in $\Irr(B)$ of height zero. Let $\w D$ be a 
defect group of $\w b$ containing $D$. Then there exists 
some $\w \nu\in\Irr( \w D)$ with $\w\chi=\kappa*\w \nu$. 
Since $\kappa_N$ is $2$-rational, $\kappa_N\in\{\chi_1,\chi_1*\delta\}$  and $\chi\in \{\chi_1*\nu, \chi_1*\nu\delta\}$, where $\chi_1$ is a $2$-rational character of 
$b$ and $\nu:=\w\nu_Z\in\Irr(D)$. 

By the definitions of $\La_{b,D}$ and $\La'_{b,D}$ we see
$\La_{b,D}(\chi), \La'_{b,D}(\chi) \in 
\{ \chi_1' *\nu, \chi_1'*(\nu\delta) \}$, 
where $\chi_1'$ is defined as in the proof of 
Proposition  \ref{prop5_13}. According to \cite[Exercise (3.10)]{Navarro}, $\chi'_1$ has an extension $\w\chi_1'\in \Irr(\NNN_G(D))$, since $\NNN_G(D)/\NNN_N(D)$ is a 
$2$-group and $\chi'_1$ can be seen as a defect 
zero character of $\NNN_N(D)/D$.

Accordingly for the proof of the statement we have to prove
that $\chi'_1*\nu$ and $\chi'_1*(\nu\delta)$ extend to $\NNN_G(D)$.
The character $\w\chi_1'*\w\nu$ is an extension of $\chi'_1*\nu$, 
see Theorem 1(ii) of \cite{Cabanes87}. 
Since $D$ is cyclic and $\nu$ non-trivial,  $\nu\delta=\nu^i$ for some
integer $i$. Hence the character $\w\nu^i$ is an extension of $\nu\delta$ and $\w\chi_1'*(\w\nu^i)$ is an extension of $\chi'_1*(\nu\delta)$.
This implies that $\chi'_1*(\nu\delta)$ extends to 
$\NNN_G(D)$. This finishes the proof. 
\end{proof}

\section{The inductive Alperin-McKay condition 
for blocks with cyclic defect groups  
}\label{sec6}

The inductive Alperin-McKay condition (or AM condition, for short) from Definition 7.2 of 
\cite{Spaeth_AM_red} can be seen as a set of properties 
satisfied for all $p$-blocks. A relative version with 
respect to $p$-groups has been introduced in 7.1 of \cite{CS12}. 
We refine this further to a condition on $p$-blocks. 
As before the inductive AM condition holds for a 
finite non-abelian simple group $S$ if it holds for $S$ 
with respect to all finite $p$-groups, and the inductive 
AM condition holds for $S$ with respect to a defect 
groups if it holds for all $p$-blocks with this specific 
defect group. It brings forth a successive approach to the 
inductive AM condition. We start by giving a blockwise version 
of the inductive AM condition.

\begin{notation}
For a $p$-block $B$ we denote by $\Irr_0(B)$ the set of 
height zero characters in $\Irr(B)$. 

Note that for a finite non-abelian simple group $S$, its 
universal covering group $G$ and its universal $p'$-covering group
the associated automorphism groups can be identified, see \cite[Corollary 5.1.4(c)]{GLS3}. 
Further for $Z\leq \Z(G)$ there is a natural embedding
of $\Aut(G/Z)$ into $\Aut(S)=\Aut(G)$. 
So it makes sense to denote by $\Aut(S)_\chi$ the stabilizer of 
$\chi$ in $\Aut(S)$ for any character $\chi\in\Irr(G)$. 
\end{notation}

In the following we state a blockwise version of 
the inductive AM condition from Definition 7.2 of  
\cite{Spaeth_AM_red}.

\begin{defi}\label{def_ind_AM_block}
Let $S$ be a finite non-abelian simple group, $G$ its universal 
covering group and $B\in \Bl(G)$ with defect group $D$. 
We say that {\it the inductive AM condition holds for $B$}, 
if the following statements hold: 
\renewcommand{\labelenumi}{(\roman{enumi})}
\renewcommand{\theenumi}{\thesubsection(\roman{enumi})}
\renewcommand{\labelenumii}{(\arabic{enumii})}
\renewcommand{\theenumii}{(\arabic{enumii})}
\begin{enumerate}
\item There exists an    
$\Aut(G)_{B,D}$-stable group $M$ 
with $\NNN_G(D)\leq M\lneq G$.
\item Let $B'\in\Bl(M)$ with $(B')^G = B$. 
There exists an $\Aut(G)_{B,D}$-equivariant bijection 
\[ \Lambda_{{B,D}}^G: \Irr_0(B)\longrightarrow  \Irr_0(B'), \]
such that 
\[ \Lambda_{B,D}^G(\Irr_0(B)\cap \Irr(G\mid \nu))
\subseteq \Irr(M\mid \nu) \text{ for every } \nu\in\Irr(\Z(G)).\]
\item \label{def6_2iii}
For every $\chi\in\Irr_0(B)$ there exists a finite group 
$A:=A(\chi)$ and characters $\w\chi$ and $\w\chi'$  
such that
\begin{enumerate}
 \item For $Z:= \ker(\chi)\cap \Z(G)$ and 
$\o G:=G/Z$ the group $A$ satisfies $\o G\lhd A$, 
$A/\Cent_A(\o G)=\Aut(G)_\chi$ and $\Cent_A(\o G) = \Z(A)$.
(More precisely $A/\Cent_A(\o G)=\Aut(G)_\chi$ makes sense since  $A/\Cent_A(\o G)$ can be identified with a subgroup of $\Aut(\o G)$ which is a subgroup of $\Aut(S)=\Aut(G)$.)
\item $\w\chi\in\Irr(A)$ is an extension of 
the character $\o\chi\in\Irr( \o G)$, where $\o\chi$ is the 
character that lifts to $\chi$.
\item  $\w\chi'\in\Irr(\o M\NNN_A(\o D))$ is                 
an extension of $\o\chi'$, where 
for $\o D:= DZ/Z$ and $\o M := M/Z$                   

the character $\o\chi'\in\Irr(\o M)$ 
is the one that lifts to      
$\chi' := \Lambda_{B,D}^G(\chi)\in\Irr_0(B')$.        
\item The characters satisfy 
\[ \Irr(\Cent_A(\o G) \mid \w\chi) =      
\Irr(\Cent_A(\o G) \mid \w\chi')\]        
and
\[ \bl(\w\chi_J) = \bl\Big( (\w\chi')_{\o M\NNN_J(\o D)}\Big)^J 
\text{ for every } J \text{ with }\o G\leq J\leq A\]
\end{enumerate}
\end{enumerate} If in the above situation the bijection $\La_{B,D}^G$ satisfies 
\[ \La_{B,D}^G(\chi)(1)_{p'}
\equiv \pm |G:\NNN_G(D)|_{p'} \,\, \chi(1)_{p'}      \mod p     
\quad \text{ for every }\chi\in\Irr(B),\]          
then we say that the {\it Isaacs-Navarro-refinement (or IN-refinement, for short) 
of the inductive AM condition holds for} $B$,
see \cite{IsaacsNavarro} and
\cite[Definitions 7.2 and 7.6]{Spaeth_AM_red}.
\end{defi}

For every $\chi\in\Irr_0(B)$ we construct, in the following, 
a finite group $A$ such that it satisfies the first two requirements of 
Part (iii) of \ref{def_ind_AM_block} are satisfied.

\begin{lem}\label{rem_exist_A_ord}
Let $p$ be a prime, $S$ a finite non-abelian simple group 
and $G$ the universal $p'$-covering group of $S$. 
Let $\chi\in\Irr(G)$ and $\o G := G/(\ker(\chi)\cap\Z(G))$. 
Then there exists a finite group $A$ with 
\renewcommand{\labelenumi}{(\roman{enumi})}
\renewcommand{\theenumi}{\thesubsection(\roman{enumi})}
\renewcommand{\labelenumii}{(\arabic{enumii})}
\renewcommand{\theenumii}{(\arabic{enumii})}
\begin{enumerate}
\item[(a)]  $\o G\lhd A$, $A/\Cent_{A}(\o G)= \Aut(G)_\chi$ 
and $\Cent_A(\o G)=\Z(A)$.
\item[(b)] The character $\o\chi\in\Irr(\o G)$ associated 
to $\chi$ extends to $A$.
\end{enumerate}
\end{lem}

\begin{proof}
The construction given in Lemma \ref{rem_exist_A} 
can easily be transferred to this situation, 
and we obtain $A$ using the same method.  
\end{proof}

In the remaining section we prove the first part of Theorem \ref{thm1_1}.
In order to be able to apply some considerations 
in future work we separate the statements that can be applied in general
from those that are specific to blocks with cyclic defect groups. 
The following statement is an analogue of Lemma \ref{lem4_2} 
for ordinary characters.

\begin{lem}\label{lem6_4_ord}
Let $N\lhd G$, $L\lhd G$ with $N\leq L$ and 
$\w\chi\in\Irr(L)$ with $\chi:=\w\chi_N\in\Irr(N)$. 
Assume that $\w\chi$ is $G$-invariant and that  
for every prime $q$ there exists an extension 
$\psi_q$ of $\chi$ to some $K_q$ with 
$(\psi_q)_{K_q\cap L}=(\w\chi)_{K_q\cap L}$, 
where $K_q$ satisfies  $N\leq K_q\leq G$ and $K_q/N\in \Syl_q(G/N)$.
Then $\w\chi$ extends to $G$. 
\end{lem}

\begin{proof} The considerations of Lemma 
\ref{lem4_2} also apply  for ordinary characters.
\end{proof}

Like in the previous section we construct 
an extension by using extensions to certain groups 
related to Sylow $q$-subgroups for primes $q \not= p$.

\begin{prop}\label{prop6_5}
Let $N \lhd G$ and  $\w\chi\in\Irr(G)$         
a character with $\chi:=\w\chi_{N} \in \Irr(N)$. Let $M\leq N$ be an 
$\NNN_G(D)$-invariant subgroup with $\NNN_N(D)\leq M$ for some 
defect group $D$ of $\bl(\chi)$. 
Suppose there exists some character $\chi'\in\Irr(M)$    
with $\bl(\chi')^N=\bl(\chi)$. For every       
prime $q$ let $G_q$ be a group such  that $N\leq G_q\leq G$ and 
$G_q/N\in \Syl_q(G/N)$.  Further let $H:=M\NNN_G(D)$ and 
$H_q:=G_q\cap H$. Let $b:=\bl(\chi)$ and $L:=G[b]$.    
Assume further that $\chi'$ has           
the following properties: 
\renewcommand{\labelenumi}{(\roman{enumi})}
\renewcommand{\theenumi}{\thesubsection(\roman{enumi})}
\begin{enumerate}
	\item For every prime $q\neq p$ there exists 
some extension $\kappa_q\in\Irr(H_q)$ of $\chi'$ such that       
	\[ \bl((\kappa_q)_{L\cap H_q})^{L\cap G_q}
           = \bl((\w \chi)_{L\cap G_q}).\]                       
	\item For $\nu\in \Irr(\Z(G)\mid \w \chi)$ there exists some extension 
$\kappa_p\in\Irr(H_p\mid \nu_{\Z(G)\cap H_p})$ of $\chi'$.       
\end{enumerate}
Then there exists some extension $\w\chi'\in\Irr(H\mid \nu)$     
of $\chi'$ such that 
\[ \bl(\w\chi'_{J\cap H})^J=\bl(\w\chi_J)\text{ for every }  
J \text{ with } N\leq J\leq G.\]
\end{prop} 

\renewcommand{\labelenumi}{(\alph{enumi})}
\renewcommand{\theenumi}{\thesubsection(\alph{enumi})}
\renewcommand{\labelenumii}{(\roman{enumii})}
\renewcommand{\theenumii}{(\roman{enumii})}

\begin{proof}
Let $\chi' \cdot \nu$ be the extension of $\chi'$ 
contained in $\Irr(M\Z(G)\mid \nu)$. 
Assumption (i) proves that $\kappa_q$ is contained in 
$\Irr(H_q\mid \nu_{\Z(G)\cap H_q})$. Hence $\kappa_q$ has an extension 
to $H_q\Z(G)$ that is an extension of $\chi'\cdot \nu$.
Assumption (ii) implies that $\chi \cdot \nu$ has an extension  
to $H_p\Z(G)$. Hence by \cite[Corollary (11.31)]{Isa} 
there exists some extension $\psi\in\Irr(H)$ of $\chi'\cdot\nu$. 

Using a construction already applied in the proof of Proposition 5.12 of \cite{NavarroSpaeth}
we define first successively a character 
$\epsilon:L\cap H \longrightarrow \CC$ with
\[ 
\bl(\psi_{J\cap H}\epsilon_{J\cap H})^J
= \bl(\w\chi_J)\text{ for every } J\text{ with } N\leq J\leq L.  
\]
For every element $x\in (L\cap H)^0$ we define 
$\epsilon^{(x)}$ to be a linear character of       
$\spann<M,x>$  with 
$M\leq \ker(\epsilon^{(x)})$ such that
\begin{align}\label{def_epsilon_x}
 \bl( \psi_{\spann<M,x>} \, \epsilon^{(x)})^{\spann<N,x>}     
= \bl(\w\chi_{\spann<N, x>}).\end{align}
Note that according to \cite[Theorem 4.1(iii)]{Murai_Dade}    
the character $\epsilon^{(x)}$ exists and is unique. 
For $x\in L\cap H$ we define $\epsilon^{(x)}$ with 
$(\epsilon^{(x)})_{\spann<M,x_{p'}>}= \epsilon^{(x_{p'})}$ and  
${\spann<M,x_{p}>} \leq \ker(\epsilon^{(x)})$.

The character $\epsilon$ is defined by 
\[ \epsilon(x)=\epsilon^{(x)}(x) \text{ for every }x\in L\cap H.\]
Note that by this definition $\epsilon$ is constant on $M$-cosets. 
Let $N\leq J\leq L$ be a group with $p\nmid |J:N|$. 
According to \cite[Theorem C(b)]{KoshitaniSpaeth1}, 
there exists a character $\delta\in\Irr(J\cap H)$ with 
$M\leq \ker(\delta)$ such that 
\begin{align*}
 \bl( \psi_{J\cap H}\delta)^{J}= \bl(\w\chi_{J}).\end{align*}
According to \cite[Lemma 2.5]{KoshitaniSpaeth1},
the character then also satisfies 
\begin{align*}
 \bl( \psi_{\spann<M,y>}\delta_{\spann<M,y>})^{\spann<N,y>}
= \bl(\w\chi_{\spann<N,y>})
\text{ for every } y\in L\cap J.\end{align*}
Since $\epsilon^{(y)}$ is uniquely defined by Equation 
\eqref{def_epsilon_x} we see that 
$\epsilon^{(y)} = \delta_{\spann<M,y>}$. By the definition of 
$\epsilon$ this implies 
\[ \epsilon_{J\cap H}=\delta_{J\cap H}.\]
Accordingly $\epsilon_E$ is a character for every group 
$E\leq L\cap H$ with $p\nmid |\spann<M,E>:M|$. 

In order to apply Brauer's characterization of characters, 
see for example Corollary (8.12) of \cite{Isa}, we have 
to consider $\epsilon_E$ for every elementary group
$E\leq L\cap H$ that is the direct product of some $p$-group 
$E_p$ and a $p'$-group $E_{p'}$. By the definition of 
$\epsilon$ the character $\epsilon\in\Irr(L\cap H)$ satisfies
\[ \epsilon(x)=\epsilon(x_p) \text{ for every } x\in L\cap H.\]
Accordingly $\epsilon_{E_{p'}}$ and hence $\epsilon$ are characters. 
The other remaining conditions from Brauer's 
characterization of characters are satisfied as well, 
since $\epsilon(1)=1$ and $(\epsilon,\epsilon)=1$ where 
$(\ ,\ )$ denotes the inner product on characters. 

Accordingly $\psi_{L\cap H} \epsilon$ is a well-defined 
character and the character satisfies by the definition of 
$\epsilon$ the equation 
\[ \bl(\psi_{J\cap H}\epsilon_{J\cap H})^J
=\bl(\w\chi_J)\text{ for every } J\text{ with } N\leq J\leq L.\]

By \cite[Theorem 3.5(i)]{Murai_Dade} (see 
\cite[Lemma 3.3 and Proposition 1.9]{DadeBlockExtensions}), 
any extension $\w\psi_1\in\Irr(H)$ 
of $\psi_1:=\psi_{L\cap H}\epsilon$ satisfies 
\[ \bl((\w\psi_1)_{J\cap H})^J
=\bl(\w\chi_J)\text{ for every } J\text{ with } N\leq J\leq G.\]
Hence it is sufficient to prove that $\psi_1$ extends to $H$. 
By Lemma \ref{lem6_4_ord} we only have to check that for 
any prime $q$ the character $(\psi_1)_{L\cap H_q}$ extends 
to $H_q$, where $H_q$ satisfies $M\leq H_q\leq H$ and 
$H_q/M\in \Syl_q(H/M)$.

For $q\neq p$ we note that $\kappa_q$ coincides with $(\psi_1)_{H_q\cap L}$ by the given construction
and hence Assumption (i) implies that $(\psi_1)_{H_q\cap L}$ 
extends to $H_q$. On the other hand 
$(\psi_1)_{H_p\cap L}=\psi_{H_p\cap L}$ 
is $H_p$-invariant and $\psi_{H_p}$ is an extension of 
$\psi_{H_p\cap L}$. Accordingly there exists an extension 
$\w\psi_1$ of $\psi_1$. By definition,                 
$\w\psi_1$ satisfies 
$\w\psi_1\in\Irr(H\mid \nu)$ and
\[ \bl((\w\psi_1)_{J\cap H})^J=\bl(\w\chi_J)
\text{ for every } J\text{ with } N\leq J\leq G.\qedhere\]
\end{proof}

\begin{thm}\label{thm6_6}
Let $p$ be an odd prime, $S$  a finite non-abelian simple group, and $G$ 
its universal covering group. Let $B\in\Bl(G)$  be a $p$-block with 
cyclic non-central defect group
$D$.                                 
Then the inductive AM condition holds for $B$. 
\end{thm}

\begin{proof}
We verify that the conditions from Definition \ref{def_ind_AM_block} 
are satisfied with $M:=\NNN_G(D)$. We take $\Lambda_{B,D}:\Irr_0(B)\longrightarrow \Irr_0(B')$ to    
be the bijection from Proposition \ref{prop4_9}. 
Then the assumptions made in the first two parts
 of Definition \ref{def_ind_AM_block} are satisfied, 
since $\Lambda_{B,D}^G$ has the required 
properties according to Proposition \ref{prop4_9}. 
It remains to check Condition (iii) from Definition \ref{def_ind_AM_block}. 

Let $\chi\in\Irr_0(B)$, $Z :=\ker(\chi)\cap\Z(G)$, $\o G:=G/Z$ and 
$\o \chi\in\Irr(\o G)$ the character associated to $\chi$, i.e. the character of $\o G$ that lifts to $G$. 
Then by Lemma \ref{rem_exist_A_ord} we can associate to $\chi$ a group $A$ 
such that $\o\chi$ extends to a character $\w\chi\in \Irr(A)$
and $A/\Cent_A(\o G)=\Aut(G)_\chi$ and $\Cent_A(\o G)=\Z(A)$. 
Further let $\o B\in \Bl(\o G)$ be the unique block of $\o G$ 
dominated by $B$. According to 
\cite[Theorem (9.9)]{Navarro} such a block exists. 
Calculations with the associated central character prove 
that this block is unique since $Z\leq \Z(G)$, see also 
\cite[Theorems 5.8.8 and 5.8.11]{NagaoTsushima}.   
For the block $\o B$ there exists a further bijection 
$\La^{\o G}_{\o B,\o D}$, see Proposition  \ref{prop4_9}.
According to Corollary \ref{cor5_2}, the character 
$\chi':=\Lambda_{B,D}(\chi)$ is a lift of 
$\o \chi':=\La_{\o B,\o D}(\o \chi)\in\Irr(\o M)$ with $\o M:=M/Z$.

Let $\w\chi\in\Irr(A)$ be the above mentioned extension of $\o\chi$. 
For every prime $q$ let $A_q$ be a subgroup with 
$\o G\leq A_q\leq A$ and $A_q/\o G\in \Syl_q(A/\o G)$. 
According to Proposition \ref{prop5_11}, 
for every prime $q\neq p$ and $\o D:=DZ/Z$ there exists an extension 
$\kappa_q\in\Irr(\NNN_{A_q}(\o D))$  of $\o \chi'$ to $A_q$ 
such that 
\[ \bl((\kappa_q)_{L\cap H_q})^{L\cap A_q}= \bl((\w \chi)_{L\cap A_q})\]
where $L := A[\bl(\o \chi)]$ and $H_q:= \NNN_A(\o D)\cap A_q$. According to Proposition
\ref{prop5_12} there exists an extension 
$\kappa_p\in\Irr(\NNN_{A_p}(\o D))$ of $\o\chi'$ 
such that for $\nu\in \Irr(\Z(A)\mid \w\chi)$ 
the character $\kappa_p$ is contained in $\Irr(\NNN_{A_p}(\o D)\mid \nu_{A_p\cap \Z(A)})$.

Now Proposition \ref{prop6_5} can be applied and proves that 
$\o \chi'$ has an extension $\w\chi'\in\Irr(\NNN_A(\o D))$ such that 
\[ \bl(\w\chi'_{J\cap \NNN_A(\o D)})^J=\bl(\w\chi_J)      
\text{ for every } J\text{ with } \o G\leq J\leq A.\]
This proves that for $\chi$ and hence for all characters 
in $\Irr_0(B)$ Part (iii) of Definition 
\ref{def_ind_AM_block} is satisfied. Accordingly the inductive 
AM condition holds for $B$. 
\end{proof}

In the situation of $p=2$ the proof is a bit more involved since 
bijections from Proposition  \ref{prop5_13} 
have less powerful properties, especially an analogue of 
Proposition \ref{prop5_12} is missing. 

\begin{thm}
Let $S$ be a finite non-abelian simple group, and $G$ its 
universal covering group. Let $B\in\Bl(G)$ be a $2$-block with 
cyclic non-central defect group. 
Then the inductive AM condition holds for $B$. 
\end{thm}
\begin{proof}
Many considerations from the proof of Theorem \ref{thm6_6} 
apply here as well. Nevertheless several adaptations are necessary.

Let $D$ be a defect group of $B$.
For the proof of the inductive AM condition let
$M :=\NNN_G(D)$ and  $B'\in\Bl(M)$ the Brauer correspondent of $B$. 
In the following we choose 
 $\La_{B,D}:\Irr(B)\longrightarrow \Irr(B')$ to be a specific
bijection from Proposition  \ref{prop5_13}. 
With both possible choices of $\La_{B,D}$ Parts (i) and (ii)
of Definition \ref{def_ind_AM_block} are satisfied 
for $B$. 

In the next step we verify Condition \ref{def6_2iii} 
for any non-$2$-rational $\chi\in\Irr(B)$. 
Let $A$ be a group associated to $\chi$ from Lemma 
\ref{rem_exist_A_ord}. The character $\chi$ extends to $A$. 
The proof of Theorem \ref{thm6_6} 
can be transferred: according to Propositions \ref{prop5_8},
 \ref{prop5_17}, and \ref{prop6_5} applies here as well 
and $\La_{B,D}(\chi)$ and $\La_{B,D}'(\chi)$ extend to 
$\NNN_A(\o D)$, where $\o D:=DZ/Z$ with $Z := \ker(\chi)\cap\Z(G)$.
Note again that this does not depend on the choice 
of the bijection between $\Irr(B)\longrightarrow\Irr(B')$. 

It remains to verify that 
Condition \ref{def6_2iii}
holds for the remaining characters. 
Let $\chi\in\Irr(B)$ be one of the two $2$-rational characters.         
According to Proposition  \ref{prop5_13}, 
$\chi$ and $\chi*\delta$ are $\Aut(G)_B$-stable.
Hence $Z := \ker(\chi)\cap\Z(G)$ contains          
the Sylow $2$-subgroup of $\Z(G)$ 
since $D\not\leq\Z(G)$.                   
Let $\o\chi\in\Irr(\o G)$ with $\o G := G/Z$ be 
the character that lifts to $\chi$.   
Let $A$ be a group as in Lemma \ref{rem_exist_A_ord} and 
let $\w\chi\in \Irr(A)$ be an extension of $\o \chi\in\Irr(\o G)$ 
such that the $2$-part of $|\ker(\w\chi)\cap\Z(A)|$ is maximal. 
(The existence of $\w\chi$ is ensured by the choice of $A$.)
Without loss of generality we can assume that $\w\chi$ is 
faithful on $\Z(A)$ otherwise one replaces $A$ by its quotient. 
For every prime $q$ let $A_q$ be a subgroup $\o G\leq A_q\leq A$ 
with $A_q/\o G\in\Syl_q(A/\o G)$. Let $C_2$ be                   
a Sylow $2$-subgroup                                                   
of $\Z(A)$. The block $\o B$ of $A_2/C_2$ dominated by
$\bl(\w \chi_{A_2})$ is nilpotent according to Lemma \ref{rem5_14}. 
Hence $\Irr(\o B)$ contains a $2$-rational  character $\psi$ of height zero, whose 
restriction $\psi_0:=\psi_{\o G}$ is $2$-rational, 
hence $\psi_0\in \{\o \chi, \o{\chi*\delta}\}$.

Assume that $\psi$ is an extension of $\o\chi$.
Then $\o \chi$ extends to $A_2/C_2$. According to \cite[Theorem (6.26)]{Isa}, 
$\o\chi$ extends to $A/C_2$ as well. Hence the group 
$\Z(A)$ is a $2'$-group. We choose $\La_{B,D}$ such                 
that $\chi\longmapsto \chi'$, where                                
$\chi'\in\Irr( B')$ is the character with $D \leq \ker(\chi')$.
Accordingly 
$\La_{B,D}( \chi)=:\chi'$ and $\chi$ extend to 
$\NNN_{A_2}(\o D)$ and $A_2$, respectively. 
Let $\w\chi\in \Irr(A)$ be 
an extension of $\o \chi\in\Irr(\o G)$.                           
Applying Propositions \ref{prop5_8} and \ref{prop6_5} 
we see that Condition \ref{def6_2iii} holds for the character $\chi$. 

Now $\chi*\delta$ can also be seen as a character 
of $\o G := G/Z$, denoted by $\o{\chi*\delta}$. 
As we have mentioned before the stabilizers of $\chi$ and
${\chi*\delta}$ in $\Aut(G)_B$ coincide, hence 
the character $\o{\chi*\delta}$ is 
$A$-invariant.                     
By \cite[Theorem 4.1]{NavarroSpaeth}         
there exists a central extension $\epsilon:
\wh A \longrightarrow A$ by a cyclic group $C$, such that $\o G$ 
can be identified with a normal subgroup 
$\o G_1\lhd \wh A$ via an isomorphism $\epsilon_0:=\epsilon_{\o G_1}:\o G_1\rightarrow \o G$
 and $\zeta =( \o\chi*\delta)\circ \epsilon_0 \in\Irr(\o G_1)$ extends 
to $\wh A$. Let $\w \zeta\in\Irr(\epsilon^{-1}(A_2)) $ be 
an extension of $\zeta$ that extends to $\wh A$. 
Then $\w \zeta$ coincides with $\w\chi*\w\delta$ for 
some extension $\w\delta\in\Irr(\w D)$ of $\delta$, where 
$\w D$ is a defect group of 
$\bl(\w\chi)$ containing $\o D := DZ/Z$. 
Now $\La_{B,D}(\chi*\delta)=\chi'*\delta$ seen as a character of 
$\NNN_G(D)/Z$ has an extension $\o \chi'*\w\delta$ 
to $\epsilon^{-1}(\NNN_{A_2}(D)/Z)$ and hence has an extension to 
$\NNN_{\wh A}(\epsilon_0^{-1}(\o D))$ with the required properties.   
According to Propositions \ref{prop6_5} Condition \ref{def6_2iii} is 
satisfied for $\chi*\delta$, as well. 

It remains to consider the case where $\psi$ 
is an extension of $\o \chi*\delta$.  
Note that according to the definition then 
$\o \chi$ and $\o{\chi*\delta}$ extend to $A$. 
Then we choose $\La_{b,D}$ such 
that $\chi*\delta\longmapsto \chi'$, where
 $\chi'\in\Irr( B')$ is the character with $D \leq \ker(\chi')$.  
Accordingly 
$\La_{B,D}(\chi*\delta)=\chi'$ and $\chi*\delta$ extend to         
$\NNN_{A_2}(\o D)$ and $A_2$, respectively. 
Let $\w\phi\in \Irr(A)$ be 
an extension of $\o {\chi*\delta}\in\Irr(A/Z)$ such that 
the Sylow $2$-subgroup of $\Z(A)$ 
is contained in the kernel of $\w\phi$. 
(This character exists because of $\psi$, see Proposition \ref{prop6_5}.)
We observe that the character of $\NNN_G(D)/D$ that 
lifts to $\chi'$ has an extension $\w\phi'$ to 
$\NNN_A(\o D)$ that contains the Sylow $2$-subgroup in its kernel.     
By Proposition \ref{prop6_5} this proves together with Proposition \ref{prop5_8} 
that Condition \ref{def6_2iii} holds for the character $\chi*\delta$. 

The character $\chi$ has by definition also an extension 
to $A$ that is faithful on $\Z(A)$.
We can write this extension again as $\w\phi*\w\delta$ for 
some extension $\w\delta\in\Irr(\w D)$ of $\o\delta$, where 
$\w D$ is a defect group of 
$\bl(\w\chi)$ containing $\o D := DZ/Z$, and              
$\o\delta\in\Irr(\o D)$ is the character                 
of $\o D$ that lifts to $\delta$.                            
The character of $\NNN_G(D)/D$ that lifts to $\chi'*\delta$ has the extension 
$\w\phi'*\w\delta$ to $\NNN_{A_2}(\o D)$.
The characters $\w\phi*\w\delta$ and $\w\phi'*\delta$ cover 
the same character of $\Z(A)\cap A_2$ by definition. By 
Proposition \ref{prop6_5} this proves together with Proposition \ref{prop5_8} 
that Condition \ref{def6_2iii} holds for the character $\chi*\delta$. 
\end{proof}

We finally mention that the IN-refinement of 
the inductive AM condition                    
for blocks with cyclic defect groups holds.    

\begin{thm}\label{thm5_8}
The inductive AM condition holds together with  
the IN-refinement 
for blocks of universal covering groups of non-abelian 
simple groups with cyclic defect groups.
\end{thm}

\begin{proof}
The bijection constructed in Proposition  \ref{prop4_9}  
coincides with the bijection used in the proof of 
Theorem 2.1 of \cite{IsaacsNavarro}.  
Hence for a block $B\in\Bl(G)$ with cyclic defect group $D$ 
there exists a bijection $\La_{B,D}^G$ satisfying 
the conditions in Definition \ref{def_ind_AM_block} and 
\[ \La_{B,D}^G(\chi)(1)_{p'}
\equiv \pm |G:\NNN_G(D)|_{p'} \,\,  \chi(1)_{p'}             
\mod p \quad \text{ for every }\chi\in\Irr(B).\]           
Since for $p=2$ the equation holds this proves the statement.
\end{proof}
As an immediate consequence we obtain the following result that 
is a generalization of Corollary 8.3(b) of \cite{Spaeth_AM_red}.
\begin{cor}
Let $S$ be a finite simple non-abelian group whose universal covering group has a cyclic Sylow $p$-subgroup. Then the inductive AM condition holds for $S$ with respect to $p$, in particular $S$ satisfies the inductive McKay conditions from Section 10 of \cite{IsaacsMalleNavarro}.
\end{cor}
Note that according to Table 2.2 of \cite{GLS3} this proves the inductive AM condition for groups of type $^3\mathrm D_4(q)$ for $p$ whenever $p^2\nmid (q^6-1)$ and $p\geq 5$, see  Corollary 7.3 of \cite{CS12}.

\section{The inductive blockwise Alperin weight condition
for blocks with cyclic defect groups   
}\label{sec7}
\noindent 
In this section we prove that the inductive blockwise Alperin 
weight condition (BAW condition, for short) holds 
for blocks with cyclic defect groups. 
We give the proof in two steps. First we prove that 
under certain additional assumptions the inductive AM condition 
for a block with abelian defect groups 
implies the inductive BAW condition for the corresponding block, 
see Theorem \ref{thm_inductiveAM_inductiveBAW}. 
Secondly we verify the second part of Theorem \ref{thm_BAW_cycl}
by applying this statement. 

It is clear that the last part of the inductive 
AM condition and the inductive BAW condition have similarities. 
Further if the considered block has abelian defect groups,
then the involved characters for the local situation belong
to the same blocks.

In order to pass from ordinary characters to 
Brauer characters we consider the decomposition matrix and 
its submatrices. Theorem \ref{thm1_2} assumes that there exists a unitriangular 
submatrix 
and states that the inductive AM condition implies the inductive BAW condition. In some sense it is a generalization of 
Theorem 3.8 in \cite{Malle_AWC_Alt} and Theorem 7.4 of \cite{CS12}, 
where weaker statements for groups of Lie type are given.

\renewcommand{\proofname}{Proof of Theorem \ref{thm1_2}}
\begin{proof}
Let  $e := |\calS|$, $\{\phi_1,\ldots ,\phi_e\}=\IBr(B)$ 
and $\{\chi_1,\ldots,\chi_e\} = \calS$ such that 
the associated decomposition matrix is unitriangular. 
Via $\chi_i\mapsto \phi_i$ this gives a bijection 
between $\calS$ and $\IBr(B)$. Furthermore there exists 
a natural correspondence between $\dz(\NNN_G(D)/D, B)$ 
and $\{ \chi'\in\Irr(B')\mid D\leq \ker(\chi')\}$ via 
lifting. Together with these bijections,
$\Lambda_B^G$ induces a bijection 
\[ \Omega^G_D:\IBr(B)\longrightarrow \dz(\NNN_G(D)/D,B).\]
Since the decomposition matrix $C'$ is unitriangular, 
$\Omega_Q^G$ is $\Aut(G)_{B,D}$-equivariant 
according to the considerations made in the proof 
of Theorem 7.4 of \cite{CS12}. 

Since the inductive AM condition holds for $B$ and for               
every character $\chi_i\in \calS$, there exists a finite group 
$A_i$ and characters $\w\chi_i$ and $\w\chi'_i$ such that              
\begin{itemize}
\item For $Z:=\ker(\chi_i)\cap\Z(G)$ and $\o G:=G/Z$ 
the group $A_i$ satisfies $\o G\lhd A_i$, 
$A_i/\Cent_{A_i}(\o G)=\Aut(G)_{\chi_i}$ and 
$\Cent_{A_i}(\o G)=\Z(A_i)$.
\item $\w\chi_i\in\Irr(A_i)$ is an extension of                     
the character $\o\chi_i\in\Irr(\o G)$ determined by $\chi_i$. 
\item For $\o D:= DZ/Z$ and $\o M:= M/Z$ let 
$\o\chi'_i\in\Irr(\o M)$ be the character defined by                  
$\Lambda_D^G(\chi_i)\in\Irr_0(M\mid D)$. 
Then $\w\chi'_i\in\Irr(\o M\NNN_{A_i}(\o D))$                       
is an extension of $\o\chi'_i$.
\item The characters satisfy 
\[ \Irr(\Cent_{A_i}(\o G) \mid \w\chi_i)                            
= \Irr(\Cent_{A_i}(\o G) \mid \w\chi'_i)\]                          
and
\[ \bl( ({\w\chi_i})_J) = 
    \bl\Big( ({\w\chi'_i})_{\o M\NNN_J(D)}\Big)^J                       
\text{ for every } J \text{ with }\o G\leq J\leq A_i.\]
\end{itemize}
We consider the situation for a fixed $\chi_i$. 
The character $\phi_i$ occurs with multiplicity one in 
$(\chi_i)^0$ and $\o\phi_i\in\IBr(\o G)$ determined by 
$\phi_i$ is a constituent with multiplicity one 
in $(\o\chi_i)^0$. Note that $\o\phi_i$ is $A_i$-invariant.               
Hence $(\w\chi_i)^0$ has a constituent $\psi\in \Irr(A_i\mid \o\phi_i)$. 
The character 
$\psi_{\o G}$ is a constituent of $(\chi_i)^0$. 
Since $\o \phi_i$ occurs with multiplicity one,
this proves $\psi_{\o G}=\o\phi_i$. Hence $\psi$ is 
an extension of $\o\phi_i$. By definition the character $\psi$ satisfies 
\[ \bl(\psi_J)=\bl(\w\chi_J) \text{ for every } J                    
\text{ with } N\leq J\leq A_i.\]
Note that because of $\Aut(S)_{\chi_i} = \Aut(S)_{\phi_i}$, 
the group $A_i$ associated to $\chi_i$ satisfies 
the properties mentioned in Definition \ref{def_BAWC}(iii)(a), 
at least after taking the quotient over the $p$-part of the center. 

By the definition of $\Omega^G_D$ the character 
$\La^G_D(\chi_i)$ is the lift of $\Omega^G_D(\phi_i)$.
Let $\theta'$ be the character of $\o M$

determined by $\La^G_D(\chi_i)$. By the inductive AM condition,
$\chi'$ has an extension $\w\chi'\in\Irr(\NNN_A(\o D)\o M)$. 
Let $\psi':=\w\chi^0$. This character is irreducible 
since $\psi'$ is an extension of $(\o\chi_i)^0$ and 
$(\o\chi_i)^0$ is irreducible. 
By the definition of $\psi$ we may conclude
\[ \bl((\psi')_{\NNN_J(D)})^J = \bl((\w\chi_i)_{\NNN_J(D)})^J
 = \bl((\w\chi_{i})_J)=\bl(\psi_J) \text{ for every } J 
\text{ with } N \leq J\leq A_i.\]

This proves that $\Omega^G_Q$ defines a bijection 
with all properties required in Definition \ref{def4_2}, 
and hence $\o B$ satisfies the inductive BAW condition, 
where $\o B$ is the block of the universal $p'$-covering 
group dominated by $B$. 
\end{proof}

We can apply the above criterion for blocks with cyclic defect groups and 
prove thereby that those blocks satisfy the inductive BAW condition.

\renewcommand{\proofname}{Proof of Theorem \ref{thm_BAW_cycl}}
\begin{proof}
Let $B$ be a $p$-block of the universal covering group $G$ of a non-abelian 
simple group $S$ with cyclic defect group. 
According to Theorem \ref{thm5_8} the inductive AM condition holds for $B$. 
For the verifications one uses the bijections $\Lambda_{B,D}^G$ from Propositions 
\ref{prop4_9} and \ref{prop5_13}. Hence the group  $M$ chosen in the 
verification coincides with $\NNN_G(D)$. 

The decomposition matrix of $B$ has been described in Theorem 
\ref{dec_matrix_unitriangular} and is unitriangular. 
Accordingly the assumptions of Theorem \ref{thm_inductiveAM_inductiveBAW} 
are satisfied and hence the inductive BAW condition holds for $B$. 
 
Let $\o B$ be a block of $\o G:= G/Z_p$ with cyclic defect group where $Z_p$ is the Sylow $p$-subgroup of $\Z(G)$. 
Note that $\o B$ might be dominated by a block $B$ of $G$ with non-cyclic defect group. 
Although we haven't proven that $B$ satisfies the inductive AM condition, the previous section 
gives a bijection with the necessary properties. Hence an adapted version of 
Theorem \ref{thm_inductiveAM_inductiveBAW} can be applied in that case and proves the statement.
\end{proof}
\renewcommand{\proofname}{Proof}

We apply Theorem \ref{thm_BAW_cycl} for some simple groups of Lie type. Note that this list is 
by far not complete, and only lists cases where the outer automorphism group is non-cyclic. 
(In the case of cyclic outer automorphism group the result is known from \cite[Corollary 8.3(b)]{Spaeth_AM_red} and \cite[Proposition 6.2]{Spaeth_red_BAW}.)
\begin{cor}\label{cor_explizit}
Let $S$ be a finite non-abelian simple group of Lie type with a non-exceptional covering and $p$ a prime dividing $|S|$ with $p\geq 5$ different from the defining characteristic of $S$. Then the inductive AM and BAW conditions hold for $S$ and $p$ in the following cases 
\begin{itemize}
\item $S$ is of type $\mathrm A_{n}(q)$, $p>n+1$ and $d$ divides only one integer in $\{2,3,\cdots, n+1 \}$, where $d$ is the order of $q$ in $(\ZZ/p\ZZ)^*$,

\item $S$ is of type $^2\mathrm A_{n}(q)$, $p>n+1$ and $d'$ divides only one integer in $\{2,3,\cdots, n+1 \}$, where $d'$ is the order of $(-q)$ in $(\ZZ/p\ZZ)^*$,

\item $S$ is of type $\mathrm B_n(q)$ or $\mathrm C_n(q)$ , $p> n$ and $d$ divides only one integer in $\{2,4,\cdots , 2n \}$, where $d$ is the order of $q$ in $(\ZZ/p\ZZ)^*$,

\item $S$ is of type $\mathrm D_n(q)$, $p> n$ and $d$ divides only one integer in $\{2,4,\cdots , 2n-2,n \}$, where $d$ is the order of $q$ in $(\ZZ/p\ZZ)^*$,
\item $S$ is of type $^2\mathrm D_n(q)$, $p> n$ and $d$ divides only one integer in $\{2,4,\cdots , 2n-2,2n \}$, where $d$ is the order of $q$ in $(\ZZ/p\ZZ)^*$,
\item $S$ is of type $\mathrm D_4(q)$, $p\nmid (q^4-1)$ and $p\geq 5$,
\item $S$ is of type $\mathrm E_6(q)$, $p\nmid (q^4-1)(q^6-1)$ and $p\geq 7$,
\item $S$ is of type $^2\mathrm E_6(q)$, $p\nmid (q^4-1)(q^6-1)$ and $p\geq 7$, and
\item $S$ is of type $\mathrm E_7(q)$, $p\nmid (q^4-1)(q^6-1)$ and $p\geq 11$.
\end{itemize}
\end{cor}
\begin{proof}
According to \cite[Corollaire 3.12]{gensyl}, together with \cite[Table 2.2]{GLS3} the universal covering group $G$ of $S$  has in the above listed cases cyclic Sylow $p$-subgroups.
Hence every $p$-block of $G$ has cyclic defect groups. Then Theorem \ref{thm_BAW_cycl} implies the statement. 
\end{proof}


{\bf Acknowledgement:} A part of this work was done while the first author was visiting the Department
of Mathematics, TU Kaiserslautern in December 2012 and October 2013. 
He is grateful to Gunter Malle for his kind hospitality and the Deutsche Foschungsgemeinschaft, SPP 1388.
The authors thank Markus Linckelmann  and Andrei Marcus 
for helpful explanations.

\def\cprime{$'$}

\end{document}